\documentclass[12 pt]{amsart}

\usepackage{hyperref}
\usepackage{etex}
\usepackage[shortlabels]{enumitem}
\usepackage{amsmath}
\usepackage{amsxtra}
\usepackage{amscd}
\usepackage{amsthm}
\usepackage{adjustbox}
\usepackage{amsfonts}
\usepackage{amssymb}
\usepackage{eucal}
\usepackage[all]{xy}
\usepackage{graphicx}
\usepackage{tikz-cd}
\usepackage{mathrsfs}
\usepackage{subfiles}
\usepackage{mathpazo}
\usepackage[colorinlistoftodos, textsize=tiny]{todonotes}
\usepackage{morefloats}
\usepackage{pdfpages}
\usepackage{thm-restate}
\usepackage[percent]{overpic}
\usepackage[utf8]{inputenc}
\usepackage{epigraph}
\usepackage{csquotes}
\usepackage[margin=1in]{geometry}
\usepackage{adjustbox}
\usepackage{microtype}
\usepackage{stmaryrd}

\usepackage{verbatim}
\usepackage{stmaryrd}
\usepackage{scalerel}
\usepackage{stackengine}
\stackMath
\newcommand\reallywidehat[1]{%
\savestack{\tmpbox}{\stretchto{%
  \scaleto{%
    \scalerel*[\widthof{\ensuremath{#1}}]{\kern-.6pt\bigwedge\kern-.6pt}%
    {\rule[-\textheight/2]{1ex}{\textheight}}
  }{\textheight}%
}{0.5ex}}%
\stackon[1pt]{#1}{\tmpbox}%
}
\usepackage{mathtools}

\graphicspath{ {images/} }

\RequirePackage{color}
\definecolor{myred}{rgb}{0.75,0,0}
\definecolor{mygreen}{rgb}{0,0.5,0}
\definecolor{myblue}{rgb}{0,0,0.65}

\usepackage{color}

\usepackage{hyperref}
\hypersetup{citecolor=blue}
\usepackage{tikz}
\usetikzlibrary{matrix,arrows,decorations.pathmorphing}

\theoremstyle{plain}
\newtheorem{theorem}[subsubsection]{Theorem}

\newtheorem{proposition}[subsubsection]{Proposition}
\newtheorem{lemma}[subsubsection]{Lemma}
\newtheorem{corollary}[subsubsection]{Corollary}

\theoremstyle{definition}
\newtheorem{definition}[subsubsection]{Definition}
\newtheorem{remark}[subsubsection]{Remark}

\newtheorem*{claim*}{Claim}
\theoremstyle{remark}

\numberwithin{equation}{section}
\newcommand\nc{\newcommand}
\nc\on{\operatorname}
\nc\renc{\renewcommand}
\DeclareMathOperator\rk{rk}

\DeclareMathOperator\id{id}
\DeclareMathOperator\Hom{Hom}

\DeclareMathOperator\GL{GL}
\DeclareMathOperator\SL{SL}

\DeclareMathOperator\gr{gr}
\DeclareMathOperator\sgn{sgn}

\nc\mf\mathfrak
\nc\mc\mathcal
\nc\mb\mathbb
\nc\msf\mathsf
\nc\mscr\mathscr

\newcommand{\defeq}{\vcentcolon=}

\title{Finite braid group orbits on $\operatorname{SL}_2$-character varieties}
\author{Yeuk Hay Joshua Lam, Aaron Landesman, and Daniel Litt}
\date{\today}

\begin{document}

\begin{abstract}
Let $\Sigma_{0,n}$ be a $2$-sphere with $n$ punctures. We classify all conjugacy
classes of Zariski-dense representations $\rho: \pi_1(\Sigma_{0,n})\to
\operatorname{SL}_2(\mathbb{C})$ with finite orbit under the mapping class group
$\operatorname{Mod}_{0,n}$, such that the local monodromy at one or more punctures has infinite order. We show that all such representations are ``of pullback type" or arise via middle convolution from finite complex reflection groups. In particular, we classify all rank $2$ local systems of geometric origin on $\mathbb{P}^1\setminus D$, with $D$ generic, and with local monodromy of infinite order about at least one point of $D$.
\end{abstract}

\maketitle

\section{Introduction}\label{section:introduction}
Fix $n>0$ 
and let $(A_1, \cdots, A_n)\in \on{SL}_2(\mathbb{C})^n$ be an $n$-tuple of
matrices such that $$\prod_{i=1}^n A_i=\on{Id}.$$ We say such an $n$-tuple is
\emph{non-degenerate} if the $A_i$ generate a Zariski-dense subgroup of
$\on{SL}_2(\mathbb{C})$. We take $X_n$ to be the set of such non-degenerate
$n$-tuples, up to simultaneous conjugation. That is, $$X_n:=\{(A_1, \cdots, A_n)\in \on{SL}_2(\mathbb{C})^n\mid \prod A_i=\on{Id} \text{ and $(A_1, \cdots, A_n)$ is non-degenerate}\}/\on{SL}_2(\mathbb{C}).$$ The set $X_n$ admits the action of a natural group of symmetries $\on{Mod}_{0,n}$, generated by $\sigma_1, \cdots, \sigma_{n-1}$, $$\sigma_i: (A_1, A_2, \cdots, A_n)\mapsto (A_1, \cdots, A_iA_{i+1}A_i^{-1}, A_{i}, \cdots A_n),$$ called the \emph{Hurwitz action}. There has been some interest over the last decades in understanding the dynamics of the $\on{Mod}_{0,n}$-action on $X_n$, as we explain in \autoref{subsection:previous-work}. Our main result is a complete classification of the finite orbits of the $\on{Mod}_{0,n}$-action on $X_n$, at least when one of the $A_i$ has infinite order.

The dynamical system described above arises naturally from the study of the character variety of $\Sigma_{0,n}$, the $2$-sphere with $n$ points removed.  The fundamental group $\pi_1(\Sigma_{0,n})$ has the presentation $$\pi_1(\Sigma_{0,n})=\langle \gamma_1, \cdots, \gamma_n \mid \prod \gamma_i=\on{id}\rangle,$$ with $\gamma_i$ a loop around the $i$-th puncture. $X_n$ is the set of complex points of the character variety parametrizing $2$-dimensional Zariski-dense $\pi_1(\Sigma_{0,n})$-representations with trivial determinant, i.e.~the categorical quotient $$X_n:=\on{Hom}(\pi_1(\Sigma_{0,n}), \on{SL}_2(\mathbb{C}))^{\text{non-deg}}\sslash \on{SL}_2(\mathbb{C}),$$ 
where the decoration ``non-deg" indicates the image of the representation is Zariski-dense. Explicitly, given a Zariski-dense representation $\rho: \pi_1(\Sigma_{0,n})\to \on{SL}_2(\mathbb{C})$, we set $A_i=\rho(\gamma_i)$.


The group $\on{Mod}_{0,n}$ is precisely the mapping class group $$\on{Mod}_{0,n}:=\pi_0(\on{Homeo}^+(\Sigma_{0,n})),$$ 
of $\Sigma_{0,n}$, acting on $X_n$ through its natural outer action on $\pi_1(\Sigma_{0,n})$. ($\on{Mod}_{0,n}$ is sometimes referred to as ``the spherical braid group on $n$ strands.")
We refer to conjugacy classes of representations with finite orbit under this action as \emph{MCG-finite}, or \emph{canonical}, representations of $\pi_1(\Sigma_{0,n})$.

\subsection{Main results}
Our main result is that such canonical representations with some $A_i$ of infinite order are of one of two types:
\begin{enumerate}
	\item The ``pullback" representations, classified by Diarra \cite[\S3-\S5]{diarra}, 
		and
	\item Representations obtained via \emph{middle convolution} from representations of finite complex reflection groups.
\end{enumerate}
See \autoref{cor:main-classification} for a precise statement. These two families are not disjoint---many ``pullback" representations also arise from middle convolution. As the finite complex reflection groups were classified by Shephard and Todd \cite{shephard-todd}, this amounts to a complete classification. We now explain the meaning of (1) and (2) above.

Note that for $n<3$, there are no Zariski-dense representations $\pi_1(\Sigma_{0,n})\to SL_2(\mathbb{C})$, as $\pi_1(\Sigma_{0,n})$ is abelian. For $n\geq 3$, let $\mathscr{M}_{0,n}$ denote the moduli scheme parametrizing genus zero curves with $n$ marked points, 
and let $\pi_n: \mathscr{M}_{0,n+1}\to \mathscr{M}_{0,n}$ be the map forgetting the last marked point. If $$\rho:\pi_1(\Sigma_{0,n})\to \on{SL}_2(\mathbb{C})$$ is an irreducible representation with finite mapping class group orbit, then there exists, 
by \cite[Corollary 2.3.5]{landesman2022canonical}, a dominant quasi-finite map $B\to \mathscr{M}_{0,n}$, and an $\SL_2$-local system $\mathbb{V}$ on $$\mathscr{C}^\circ:=\mathscr{M}_{0,n+1}\times_{\mathscr{M}_{0,n}} B,$$ with the following property: if $X$ is a fiber of the projection $\mathscr{C}^\circ\to B$, then $\mathbb{V}|_X$ has monodromy $\rho$. Thus it suffices to classify rank $2$ local systems on such $\mathscr{C}^\circ$.

By results of Corlette-Simpson \cite[Theorem 2]{corlette-simpson} (when $\mathbb{V}$ has quasi-unipotent monodromy at infinity, see \autoref{defn:quasiunipotent}) 
and Loray-Pereira-Touzet \cite[Theorem A]{loray-etc} (otherwise),
such local systems with Zariski-dense monodromy come in two flavors: those pulled back through a map $\mathscr{C}^\circ\to Y$, with $Y$ a Deligne-Mumford curve, and those pulled back from a so-called polydisk Shimura variety. The former are, by definition, the ``pullback" representations classified by Diarra \cite[\S3-\S5]{diarra}, so all that remains is to classify the latter---equivalently, those that underlie a rank $2$ integral variation of Hodge structure, or are equivalently (in rank $2$, see \autoref{rmk:geometric-origin}) of geometric origin, i.e.~arise in the cohomology of a family of smooth proper varieties over $\mathscr{C}^\circ$. A local system of geometric origin on a \emph{generic} $n$-times punctured curve of genus zero automatically spreads out over moduli space---thus it suffices to classify these. Here $(\mathbb{P}^1_{\mathbb{C}}, x_1, \cdots, x_n)$ is \emph{generic} if the corresponding map $\on{Spec}(\mathbb{C})\to \mathscr{M}_{0,n, \mathbb{Q}}$ factors through the generic point.

\begin{definition}
	An element $g\in \on{GL}_r(\mathbb{C})$ is a \emph{pseudoreflection} if
	$g-\on{Id}$ has rank $1$. A finite subgroup of $\on{GL}_r(\mathbb{C})$ is a \emph{finite complex reflection group} if it is generated by pseudoreflections.
\end{definition}

This is our main technical result:
\begin{theorem}\label{thm:main-thm-intro}
Let $D\subset \mathbb{P}^1$ be a generic reduced effective divisor of degree $n$, containing $\infty$, and let $\mathbb{V}$ be a local system of rank $2$ of geometric origin	on $X=\mathbb{P}^1\setminus D$ with non-scalar monodromy at $\infty$. Suppose that $\mathbb{V}$ has Zariski-dense monodromy in $\on{SL}_2(\mathbb{C})$ and at least one local monodromy matrix of $\mathbb{V}$ has infinite order. Then there exists a local system $\mathbb{U}$ on $X$ with monodromy group a finite complex reflection group, a rank one, finite order local system $\mathbb{M}$ on $X$, and a rank one, finite order local system $\chi$ on $\mathbb{G}_m$, such that $\mathbb{V}=MC_\chi(\mathbb{U})\otimes \mathbb{M}.$
\end{theorem}

Here $MC_\chi$ denotes Katz's middle convolution operation \cite{katz-rigid}, which will be described in \autoref{section:mc}. The proof, given in \autoref{section:intro-proofs}, relies on a Hodge-theoretic analysis of middle convolution. 
\begin{remark}
One may interpret the result as saying that $\mathbb{V}\otimes \mathbb{M}^{-1}$ appears in the cohomology of a family of curves over $X$, whose fiber over $x\in X$ is a $(\mathbb{Z}/r\mathbb{Z}\times G)$-cover of $\mathbb{P}^1$, branched over $\{x\}\cup D$, where $r$ is some positive integer and $G$ is a finite complex reflection group. The cases with $G$  a dihedral group are precisely those studied in \cite{lamlitt}, which gave us some hope that \autoref{thm:main-thm-intro} could be true.
\end{remark}
As we will see, \autoref{thm:main-thm-intro} gives quite fine control of the local systems $\mathbb{U}, \chi, \mathbb{M}$; for example, the dimension and local monodromies of $\mathbb{U}$ are given explicitly in \autoref{prop:katz-output}, and their monodromy representations may be computed algorithmically from that of $\mathbb{V}$ (and vice versa) \cite{dettweiler-reiter-painleve}.

\begin{remark} In fact we prove a slightly more general result, \autoref{thm:main-technical-thm}, where the condition that some local monodromy matrix have infinite order is relaxed in favor of requiring that $\mathbb{V}$ has no unitary Galois conjugates. Note that this latter condition is automatic if some local monodromy matrix has infinite order. Indeed, local systems of geometric origin always have quasi-unipotent local monodromy, i.e.~all eigenvalues are roots of unity. Thus if there is a local monodromy matrix of infinite order, it must not be diagonalizable, and hence cannot be unitary.

The condition that $\mathbb{V}$ has no unitary Galois conjugates has attracted some interest in the case of representations in $X_3$ with quasi-unipotent monodromy at infinity, see \cite{mcmullen-triangles, mcmullen-hilbert}.
\end{remark}
\begin{remark}
In the statement of \autoref{thm:main-thm-intro}, the condition of non-scalar monodromy  at infinity is just a notational convenience; as $\mathbb{V}$ has Zariski-dense monodromy, at least one of the punctures must have non-scalar monodromy. So this hypothesis can always be achieved by relabeling the points of $D$.
\end{remark}
\begin{remark}\label{rmk:geometric-origin}
	From the point of view of local systems of geometric origin, the
	condition that some $A_i$ have infinite order is natural. By
	Corlette-Simpson \cite[Theorem 2 and \S9]{corlette-simpson}, rank $2$ local
	systems of geometric origin with Zariski-dense monodromy always arise in the cohomology of a
	generically simple Abelian scheme ``of $\on{GL}_2$-type." The condition that some $A_i$ have infinite order is precisely the condition that this Abelian scheme  have some point of potentially totally degenerate reduction, i.e.~there is a point at which the abelian scheme admits, after a finite extension of the corresponding discretely valued field, a semistable model with toric special fiber.  
		
	In this way our main technical result, \autoref{thm:main-thm-intro},
	classifies all non-isotrivial $\on{GL}_2$-type Abelian schemes over $\mathbb{P}^1$ with:
	\begin{enumerate}
		\item at least one point of potentially totally degenerate reduction, and
		\item generic discriminant locus.
	\end{enumerate}
\end{remark}
The following classification result for MCG-finite representations is deduced in \autoref{section:intro-proofs}:
\begin{corollary}\label{cor:main-classification}
	Let $\rho: \pi_1(\Sigma_{0,n})\to SL_2(\mathbb{C})$ be a Zariski-dense MCG-finite representation with at least one local monodromy matrix of infinite order. Then either
	\begin{enumerate}
		\item $\rho$ is a ``pullback" representation (classified by Diarra \cite[\S3-\S5]{diarra}), or
		\item $\rho=MC_\chi(\gamma)\otimes \nu$, where $\nu:
			\pi_1(\Sigma_{0,n})\to \mathbb{C}^*$ and $\chi:
			\pi_1(\mathbb{G}_m)\to \mathbb{C}^*$ are finite order
			characters, and $\gamma: \pi_1(\Sigma_{0,n})\to \GL_r(\mathbb{C})$ is an irreducible representation with image a finite complex reflection group. 
	\end{enumerate}
\end{corollary}

In \autoref{section:fcrg} we analyze all finite complex reflection groups of rank at least $5$, and in particular we show in \autoref{gmp5} that local systems as in \autoref{thm:main-thm-intro} cannot have non-scalar monodromy at more than $6$ points. Combining this analysis with Diarra's result \cite[Th\'eor\`eme 5.1]{diarra}, which rules out ``pullback" representations in this regime, we deduce the following corollary in \autoref{section:D-large}.
\begin{corollary}\label{cor:finite-orbits-7}
	Let $$\rho: \pi_1(\Sigma_{0,n})\to \on{SL}_2(\mathbb{C})$$ be a representation with Zariski-dense image and non-scalar local monodromy at each puncture. Suppose the local monodromy at one or more punctures has infinite order, and that $[\rho]\in X_n$ has finite orbit under the action of $\on{Mod}_{0,n}$. Then $n\leq 6$.
\end{corollary}

In \autoref{section:examples}, we give a number of examples of such local systems on $\Sigma_{0,n}$ with non-scalar monodromy at $n\leq 6$ points, arising via middle convolution from finite complex reflection groups, which show that the bound of \autoref{cor:finite-orbits-7} is sharp. \autoref{cor:finite-orbits-7} proves part of the conjecture of \cite[\S11]{tykhyy-garnier} in the case under consideration, i.e.~when there is a local monodromy matrix of infinite order and $n> 6$.

\subsection{Previous work}\label{subsection:previous-work} 
The action of $\text{Mod}_{0,n}$ on $X_n$ has been heavily studied, in several quite distinct contexts. While we cannot summarize all of this work here, we do our best to indicate the various perspectives motivating research in the subject.
\subsubsection{Dynamics of character varieties}
Fixing conjugacy classes $C_1, \cdots, C_n\subset \on{SL}_2(\mathbb{C})$ and, setting $\mathcal{Z}=(C_1, \cdots, C_n)$, 
we define the \emph{relative character variety} $X_n(\mathcal{Z})$ of $\pi_1(\Sigma_{0,n})$ to be the subset of $X_n$ consisting of those conjugacy classes of $(A_1, \cdots, A_n)$ such that $A_i\in C_i$ for all $i$. As the action of $\on{Mod}_{0,n}$ simply permutes the $C_i$, there is an index $n!$ subgroup $\on{PMod}_{0,n}\subset \on{Mod}_{0,n}$ whose action preserves $X_n(\mathcal{Z})$. The dynamics of this action, and its higher genus analogues, are extremely well-studied, though still far from completely understood. Our results amount to a classification of finite orbits in $X_n(\mathcal{Z})$ 
when some $C_i$ has infinite order.

The study of $X_4$ goes back to Fricke and Klein \cite{fricke-klein}, and the study of its dynamics dates to work of Markoff \cite{markoff1879formes, markoff1880formes}. In this case $X_4(\mathcal{Z})$ 
is an open subset 
of the affine cubic surface $$x^2+y^2+z^2+xyz=ax+by+cz+d;$$  the dependence of $a,b,c,d$ on $\mathcal{Z}$ may be found in \cite[Equation (1.5)]{cantat2009holomorphic}, where Cantat and Loray studied the holomorphic dynamics of the $\on{Mod}_{0,4}$-action on $X_4$ in some detail. Goldman \cite{goldman1997ergodic, goldmanmapping}, Previte-Xia \cite{previte2000topological, previte2002topological}, Pickrell-Xia \cite{pickrell2002ergodicity, pickrell2003ergodicity} and others studied the action of the mapping class group of a surface $\Sigma_{g,n}$ on its character variety (valued in a compact Lie group) with the aim of understanding ergodicity and density properties of this action. Brown, Eskin, Filip, and Rodriguez-Hertz have ongoing work studying this action with the goal of understanding orbit closures and invariant measures.  
Bourgain, Gamburd, and Sarnak, \cite{bgs1, bourgain2016markoff} as well as Chen \cite{chen2020nonabelian}, used the dynamics of the $\on{Mod}_{0,4}$-action on $X_4$ to prove strong approximation results for $X_4$; Whang \cite{whang2020arithmetic, whang2020global, whang2020nonlinear} began the process of generalizing these results to higher genus and $n$ large. See also Michel's work \cite{michel2006hurwitz}, which amounts to an analysis of the finite orbits of the action of $\on{Mod}_{0,n}$ on the real points of certain higher rank relative character varieties, parametrizing representations into (not necessarily finite) reflection groups.

Finally, two of the authors of this paper studied the action of the mapping class group $\on{Mod}_{g,n}$ of a higher genus surface $\Sigma_{g,n}$ on the character variety parametrizing $r$-dimensional representations of $\pi_1(\Sigma_{g,n})$, and showed that if $g>r^2-1$, all finite orbits correspond to representations with finite image \cite{landesman2022canonical}, after an analogous result was proved by Biswas-Gupta-Mj-Whang in rank $2$ \cite{biswas-whang}.
\subsubsection{Isomonodromy}
Much previous work on the question considered here has been done in the context of \emph{isomonodromy}. In particular classifying finite orbits of the $\on{Mod}_{0,n}$-action on $X_n$ when $n=4$ is equivalent to classifying algebraic solutions to the Painlev\'e VI equation \cite{lisovyy-tykhyy}; this is the primary source of historical interest in the question. Our work in this paper may be considered analogously---our main result classifies algebraic solutions to the Schlesinger system
$$
    \begin{dcases}
      \frac{\partial B_i}{\partial \lambda j}=\frac{[B_i, B_j]}{\lambda_i-\lambda_j} & i\neq j\\
      \frac{\partial B_i}{\partial \lambda_i}=-\sum_{j\neq i}\frac{[B_i, B_j]}{\lambda_i-\lambda_j}\\
      \sum B_i = 0
    \end{dcases}
$$
where the $B_i\in \mathfrak{sl}_2(\mathbb{C})$, $i=1,\cdots, n$, with
$\on{exp}(2\pi i B_j)$ of infinite order for some $j$; algebraic solutions to
this system correspond to finite orbits of $\on{Mod}_{0,n}$ on $X_n$ (see
\cite[Theorem A]{cousin2017algebraic} 
or \cite[Remark 2.3.6]{landesman2022canonical}).

The history of the case of $X_4$---the Painlev\'e VI equation---is summarized nicely by Boalch \cite{boalch2007towards}. Finite orbits (equivalently, algebraic solutions to the Painlev\'e VI equation) were discovered by many people, including Andreev-Kitaev \cite{andreev2002transformations}, Boalch \cite{boalch-higher, boalch-icosahedral,boalch-six, boalch2005klein, boalch2007towards}, Doran \cite{doran}, Kitaev \cite{kitaev2005remarks, kitaev2006grothendieck}, Dubrovin-Mazzocco \cite{dubrovin-mazzocco}, Hitchin \cite{hitchen-poncelet}, and others. The list of finite orbits they found was eventually proven to be complete by Lisovyy \& Tykhyy, using a computer-aided proof \cite{lisovyy-tykhyy}.

Finite orbits of the $\on{Mod}_{0,5}$-action on $X_5$ were studied by
Calligaris-Mazzocco \cite{calligaris-mazzocco} and Tykhyy \cite{tykhyy-garnier}, who gave a complete, computer-aided classification in this case. Beyond this, no general classification was known prior to this work, as far as we are aware.

\subsubsection{Middle convolution} The use of middle convolution and finite complex reflection groups to produce solutions to the Painlev\'e VI equation, where middle convolution is closely related to the ``Okamoto transform" \cite[\S9]{loray2010foliations}, has been known for some time; it appears implicitly in Boalch's work \cite{boalch-reflections, boalch2005klein} and explicitly in work of Dettweiler-Reiter \cite{dettweiler-reiter-painleve}. Girand \cite[\S5]{girand2016equations} used middle convolution to produce some finite braid group orbits in $X_5$. Our Hodge-theoretic analysis of middle convolution has precursors in \cite{dettweiler-sabbah, dettweiler-sabbah-erratum, nicolas-mc}; in principle we expect we could have used their results to prove \autoref{prop:key-computation}, though we chose to give a direct argument.

Kitaev conjectured \cite{kitaev-special, kitaev2006grothendieck} that all finite orbits in the case of $X_4$ (namely, algebraic solutions to the Painlev\'e VI equation) are, up to Okamoto transformation/middle convolution, either ``pullback representations" or have finite monodromy; this conjecture was confirmed by Lisovyy-Tykhyy's classification \cite{lisovyy-tykhyy}. It is interesting to consider our results here in light of this conjecture: our work gives a computer-free proof in the case some local monodromy matrix has infinite order, and a generalization to the case that $n>4$.
\subsection{Sketch of proof}
Our proof of \autoref{thm:main-thm-intro} is closely related to Katz's classification of rigid local systems \cite{katz-rigid}. Given an irreducible rigid local system $\mathbb{V}$ of rank at least $2$ on $\mathbb{P}^1\setminus D$, Katz gives an algorithm for choosing a rank one local system $\mathbb{L}$ on $\mathbb{P}^1\setminus D$, and a rank one local system $\chi$ on $\mathbb{G}_m$, such that $MC_\chi(\mathbb{V}\otimes \mathbb{L})$ is rigid and has rank less than that of $\mathbb{V}$. Hence iterating this procedure eventually yields a local system of rank one. As middle convolution is an invertible operation, this shows that all rigid local systems can be constructed out of local systems of rank one.

We show in \autoref{section:katz} that a variant of Katz's algorithm, applied to a local system satisfying the hypotheses of \autoref{thm:main-thm-intro}, necessarily yields a local system with all of its Galois conjugates \emph{unitary}. Combined with the fact that Katz's algorithm preserves integrality of a local system, this implies that the local system thus produced has finite monodromy. We conclude that its monodromy is given by a finite complex reflection group by analyzing the effect of middle convolution on local monodromy. Unlike in Katz's situation, our application of middle convolution typically \emph{increases} the rank of the local system in question.

It is natural to ask: to what extent does middle convolution explain MCG-finite local systems in higher rank? Rigid local systems are evidently MCG-finite (as they are isolated points of relative character varieties, and hence are permuted by $\on{Mod}_{0,n}$). The fact that a variant of Katz's algorithm extends beyond the rigid case to classify some MCG-finite local systems in our setting is, we think, suggestive.

\subsection{Notation} \label{section:notation} A \emph{family of $n$-pointed projective lines} is a map $\pi: \mathbb{P}^1\times \mathscr{M}\to \mathscr{M}$ equipped with $n$ disjoint sections $s_1, \cdots, s_n$. The \emph{associated family of $n$-punctured projective lines} $\pi^\circ: \mathscr{C}^\circ\to \mathscr{M}^\circ$ is the restriction of $\pi$ to the complement of the images of the $s_i$. For $n\geq 3$ we say that $\pi$ is a \emph{versal family} if the induced map $\mathscr{M}\to \mathscr{M}_{0,n}$ is dominant, and in this case we refer to $\pi^\circ$ as the associated \emph{punctured versal family}. We work over the complex numbers throughout.
\subsection{Acknowledgments}
We are grateful for many useful conversations with Philip Boalch, Serge Cantat, Bertrand Deroin, Alex Eskin, Bruno Klingler, Oleg Lisovyy, Frank Loray, Claude Sabbah, Will Sawin, Carlos Simpson, and Nicolas Tholozan. Lam was supported by a Dirichlet Fellowship at Humboldt University, Berlin during the course of this work. Litt was supported by the NSERC Discovery Grant, "Anabelian methods in arithmetic and algebraic geometry."
\section{Parabolic Higgs bundles, variations of Hodge structure, and isomonodromy}\label{section:parabolic}
In this section we recall some preliminaries on parabolic Higgs bundles, non-abelian Hodge theory for non-compact Riemann surfaces, and isomonodromy. 
\subsection{Preliminaries on parabolic bundles}
Let $C$ be a smooth proper curve, and let $D=x_1+\cdots +x_n$ be a reduced divisor on $C$. A parabolic bundle $E_\star$ on $(C, D)$ is a vector bundle $E$ on $C$, equipped with a decreasing filtration $E_{x_j} =E^1_j \supsetneq E^2_j\supsetneq \cdots \supsetneq E^{n_j+1}_{j}=0$ at each $x_j$, with attached  weights $0\leq \alpha_j^1< \cdots < \alpha_{j}^{n_j}<1$. See \cite[\S2]{landesman2022geometric} for further background, notions of homomorphisms of parabolic bundles, quotients and subbundles, slope, stability, etc. We will need the following specialization of the definition of parabolic Hom sheaves in \cite[\S2.2]{landesman2022geometric}:

If $E_\star, F_\star$ are parabolic line bundles on $(C, D)$, the sheaf of parabolic homomorphisms is 
\begin{align}
	\label{equation:parabolic-0-degree}
	\underline{\Hom}(E_\star, F_\star)_0=\underline{\Hom}(E, F)(-D'),
\end{align}
where $D'\subset D$ is the set of points $x_j$ at which the weight of $E_{x_j}$ is greater than that of $F_{x_j}$.

If $(E, \nabla: E\to E\otimes \Omega^1_C(\log D))$ is a flat vector bundle on $C$ with logarithmic singularities along $D$, then $E$ naturally obtains the structure of a parabolic bundle as in \cite[Definition 3.3.1]{landesman2022geometric}. If $x$ is a point of $D$, let $M$ be the residue matrix of $\nabla$ at $x$. Then the weights of the parabolic structure of $E$ at $x$ are $$\on{Re}(\lambda)-\lfloor \on{Re}(\lambda)\rfloor,$$ where $\lambda$ runs over the eigenvalues of $M$.
\subsection{Preliminaries on local systems and variations of Hodge structure}
Recall that a  complex variation of Hodge structure on $C\setminus D$ is a triple $(V, V^{p,q}, \mathbf{D})$ where 
\begin{itemize}
\item $V$ is a $C^{\infty}$ complex vector bundle on $C\setminus D$, equipped with a decomposition $V=\oplus V^{p,q}$, 
\item $\mathbf{D}$ is a flat connection on $V$ satisfying Griffiths transversality, i.e. $\mathbf{D}(V^{p,q})\subset A^{1,0}(V^{p,q}\oplus V^{p-1, q+1})\oplus A^{0,1}(V^{p,1}\oplus V^{p+1, q-1}),$ where $A^{i,j}(E)$ denotes the sheaf of $E$-valued $(i,j)$-forms.
\end{itemize}
It is moreover polarizable if $V$ can be equipped with a Hermitian form $\psi$ flat with respect to $\mathbf{D}$, and such that $V^{p,q}$ and $V^{p', q'}$ are orthogonal to each other for $(p,q)\neq (p',q')$, and $(-1)^p\psi$ is positive definite on $V^{p,q}$. 

Mehta and Seshadri \cite{mehta-seshadri} famously give a correspondence between polystable parabolic bundles of slope 
zero on a pair $(C,D)$ and unitary local systems on $C\setminus D$. The following is a standard generalization of their theorem to polarizable complex variations of Hodge structure $(V, V^{p,q}, \mathbf{D})$ on $C\setminus D$.
\begin{proposition}\label{prop:NAHT}
\begin{enumerate}
\item $(E, \nabla)\defeq (\ker(\mathbf{D})\otimes \mc{O}, \mathrm{id}\otimes d)$ is a holomorphic flat vector bundle, and is moreover equipped with the Hodge filtration $F^pV\defeq \oplus V_{j\geq p}^{j,q}$ by holomorphic sub-bundles. We refer to it as the \emph{holomorphic flat vector bundle} associated to $(V, V^{p,q}, \mathbf{D})$.
\item
There exists an unique flat bundle $(\overline{E}, \overline{\nabla})$ on $C$ with logarithmic singularities along $D$ extending $(E, \nabla)$, known as the    Deligne canonical extension, such that the eigenvalues of the residues of $\overline{\nabla}$ have real parts in $[0,1)$. Moreover,  $(\overline{E}, \overline{\nabla})$ has a Hodge filtration extending that on $E$, which we also denote by $F^{\bullet}$.
\item $\overline{E}$ may be given the structure of a parabolic bundle as in \cite[Definition 3.3.1]{landesman2022geometric}, which we denote by $\overline{E}_{\star}$. The associated graded $(\oplus_i \gr^i_{F^\bullet} \overline{E}_{\star}, \theta)$ of $(\overline{E}_{\star}, F^{\bullet}, \overline{\nabla})$ is a Higgs bundle, and is moreover parabolically polystable of slope zero. 
\end{enumerate} 
\end{proposition}
\begin{proof}
Point (1) is standard and straightforward to check. The rest of the proposition is a form of \cite[Theorem 8]{simpson1990harmonic}. The construction of the Deligne canonical extension is given  for example in \cite[Definition 4.1.2]{landesman2022geometric}.  For the construction of the parabolic structure on $\overline{E}$, see \cite[Definition 3.3.1]{landesman2022geometric}. For the claim in (2) that the Hodge filtration extends, as well as the semistability statement in (3), we refer the reader to \cite[Proposition 4.1.4]{landesman2022geometric} and the references therein. 
\end{proof}

\begin{definition}\label{defn:quasiunipotent}
	Let $X$ be a smooth quasiprojective variety and $\mathbb{V}$ a local system on $X$. We say that $\mathbb{V}$ has \emph{quasi-unipotent monodromy at infinity} if there exists a simple normal crossings compactification $X\subset \overline{X}$ so that the local monodromy about each component of the boundary $\overline{X}\setminus X$ is \emph{quasi-unipotent}, i.e.~all eigenvalues are roots of unity.
\end{definition}

\subsection{Isomonodromy} 
We let $D$ be a reduced effective divisor of degree $n$ on $\mathbb{P}^1$. Let $(\mathscr{E},\nabla: \mathscr{E}\to \mathscr{E}\otimes \Omega^1(\log D))$ be a logarithmic flat vector bundle on $(\mathbb{P}^1, D)$. 

Suppose we are given a contractible complex manifold $B$, $o\in B$ a point, and a divisor $\mathscr{D}\subset  \mathbb{P}^1\times B$, \'etale over $B$. Suppose moreover we have a fixed isomorphism $(\mathbb{P}^1_o, \mathscr{D}_o)\simeq (\mathbb{P}^1, D)$. Then there is a canonical flat vector bundle $$(\widetilde{\mathscr E}, \widetilde{\nabla}: \widetilde{ \mathscr E}\to \widetilde{\mathscr E}\otimes \Omega^1_{\mathbb{P}^1\times B}(\log \mathscr D))$$ equipped with an isomorphism  $(\widetilde{ \mathscr E}, \widetilde{\nabla})|_{\mathbb{P}^1_o}\overset{\sim}{\to}(\mathscr{E}, \nabla)$, see \cite[\S3.4]{landesman2022geometric} or \cite[Th\'eor\`eme 2.1]{malgrange}, referred to as the \emph{isomonodromic deformation} of $(\mathscr{E}, \nabla)$. This is etymologically justified because the monodromy representation associated to $(\widetilde{\mathscr E}, \widetilde{\nabla})|_{\mathbb{P}^1\setminus \mathscr{D}_b})$ is independent of $b\in B$.

Given $(\mathscr{E}, \nabla)$ as above, with $n\geq 3$, we may always take $B$ to be the universal cover of the moduli space $\mathscr{M}_{0,n}$ of $n$-pointed curves of genus zero. 
In this case we refer to the fiber of $(\widetilde{\mathscr E}, \widetilde \nabla)$ over a general $b\in B$ as \emph{the isomonodromic deformation of $(\mathscr{E}, \nabla)$ to a general nearby $n$-pointed curve}. Here, by \emph{general} we mean ``on the complement of a nowhere dense analytic subset." (See \cite[Definition 3.4.4]{landesman2022geometric} and the surrounding text for a discussion of this notion.) 
\begin{proposition}\label{prop:unobstructed-destabilizing-bundle}
	Let $(\mathscr{E}, \nabla: \mathscr{E}\to \mathscr{E}\otimes\Omega^1_{\mathbb{P}^1}(\log D))$ be a logarithmic flat vector bundle of rank $2$ on $(\mathbb{P}^1,D)$ with irreducible monodromy. Suppose that the isomonodromic deformation $(\mathscr{E}', \nabla')$ of $(\mathscr{E}, \nabla)$ to a general nearby $n$-pointed curve $(\mathbb{P}^1, D')$ (with parabolic structure as in \cite[Definition 3.3.1]{landesman2022geometric}) is not parabolically semistable, with maximal destabilizing sub-bundle $F_\star$. Then $\underline{\Hom}(F_\star, \mathscr{E}'/F_\star)_0=\mathscr{O}(-1)$.
\end{proposition}
\begin{proof}
	This follows from a deformation-theoretic argument, as in the proof of \cite[Theorem 6.1.1]{landesman2022geometric}. The connection $\nabla'$ yields a nonzero $\mathscr{O}$-linear map $T_{\mathbb{P}^1}(-D)\to \underline{\Hom}(F_\star, \mathscr{E}'/F_\star)_0$ (that the map is non-zero follows as the monodromy of $(\mathscr{E}', \nabla')$ is irreducible). As this is a map of line bundles, the cokernel is torsion, and hence the map is surjective on $H^1$.
	
	As $D'$ is general, \cite[Lemma 6.4.2]{landesman2022geometric} implies that this map induces the zero map $$H^1(T_{\mathbb{P}^1}(-D'))\to H^1(\underline{\Hom}(F_\star, \mathscr{E}'/F_\star)_0).$$ Thus in particular we must have  $H^1(\underline{\Hom}(F_\star, \mathscr{E}'/F_\star)_0)=0$. But as $F$ is maximal destabilizing, the parabolic degree of $\underline{\Hom}(F_\star, \mathscr{E}'/F_\star)$ (and hence the honest degree of $\underline{\Hom}(F_\star, \mathscr{E}'/F_\star)_0$) is negative; thus $\underline{\Hom}(F_\star, \mathscr{E}'/F_\star)_0$ must be the unique-up-to-isomorphism line bundle on $\mathbb{P}^1$ with negative degree and vanishing $H^1$, namely $\mathscr{O}(-1)$.
\end{proof}
\begin{proposition}\label{prop:not-all-unipotent}
	Let $\mathbb{V}$ be an irreducible complex local system of rank $2$ on
	$\mathbb{P}^1\setminus D$ and let $(\mathscr{E}, \nabla: \mathscr{E}\to
	\mathscr{E}\otimes\Omega^1_{\mathbb{P}^1}(\log D))$ be its Deligne
	canonical extension. Suppose that the isomonodromic deformation
	$(\mathscr{E}', \nabla')$ of $(\mathscr{E}, \nabla)$ to a general nearby
	$n$-pointed curve $(\mathbb{P}^1, D')$ underlies a polarizable complex
	variation of Hodge structure. Then there exists a point $x$ of $D$ such
	that  $\mathbb{V}$ does not have unipotent monodromy at $x$.
\end{proposition}
\begin{proof}
Assume to the contrary that $\mathbb{V}$ has unipotent monodromy at all points $x$ of $D$. If $\mathbb{V}$ is unitary, then in fact $\mathbb{V}$ has trivial monodromy at all points $x$ of $D$, contradicting its irreducibility. So we may assume $\mathbb{V}$ is not unitary, i.e.~after shifting indices, it underlies a variation of Hodge structure with $h^{1,0}=h^{0,1}=1$, and $h^{i,j}=0$ otherwise. Let $F\subset \mathscr{E}$ be the non-trivial piece of the Hodge filtration.

As $\mathbb{V}$ has unipotent monodromy along $D$, the induced parabolic structure on $\mathscr{E}$ is trivial, so the Higgs bundle $\oplus_i \on{gr}^i_{F^\bullet}(\mathscr{E})=F\oplus \mathscr{E}/F$ with the Higgs field induced by $\nabla$ is stable as a Higgs bundle of slope zero, by \autoref{prop:NAHT}(3). As $F$ is a quotient Higgs bundle, it must have positive degree, so we have $\deg F\geq 1, \deg \mathscr{E}/F\leq -1$. Thus $\deg \underline{\Hom}(F, \mathscr{E}/F)\leq -2$, contradicting \autoref{prop:unobstructed-destabilizing-bundle}.
\end{proof}

\section{Middle convolution}\label{section:mc}
We now recall the definition and basic properties of Katz's middle convolution operation. Aside from Katz's book \cite{katz-rigid}, the reader may find \cite{dettweiler-reiter-painleve, simpson2009katz} to be useful references.
\subsection{Definitions} As before, we let $D\subset \mathbb{P}^1$ be a reduced effective divisor containing $\{\infty\}$, and set $X=\mathbb{P}^1\setminus D$. We let $U=X\times X\setminus \Delta$, where $\Delta$ is the diagonal, and let $j: U\hookrightarrow \mathbb{P}^1\times X$ be the natural inclusion. Let $\pi_1: U\to X$ be the projection onto the first coordinate, and let $\pi_2: \mathbb{P}^1\times X\to X$ be the projection onto the second coordinate. Let $\alpha: U\to \mathbb{G}_m$ be the map $(x, y)\mapsto x-y$. Some of this data is depicted in \autoref{figure:mc}.
\begin{definition}
	Given a local system $\mathbb{W}$ on $X$, and a rank one local system $\chi$ on $\mathbb{G}_m$, the middle convolution $MC_\chi(\mathbb{W})$ is defined to be $$MC_\chi(\mathbb{W}):=R^1\pi_{2*}j_*(\pi_1^*\mathbb{W}\otimes \alpha^*\chi).$$
\end{definition}
This definition is a translation of Katz's \cite[\S2.8]{katz-rigid}, avoiding the language of perverse sheaves. See \cite[Theorem 1.1]{dettweiler2003middle} for a proof these two definitions agree.
\begin{figure}[h]
\centering{
\resizebox{110mm}{!}{
\begin{overpic}[trim={0 12cm 0 0}, clip]{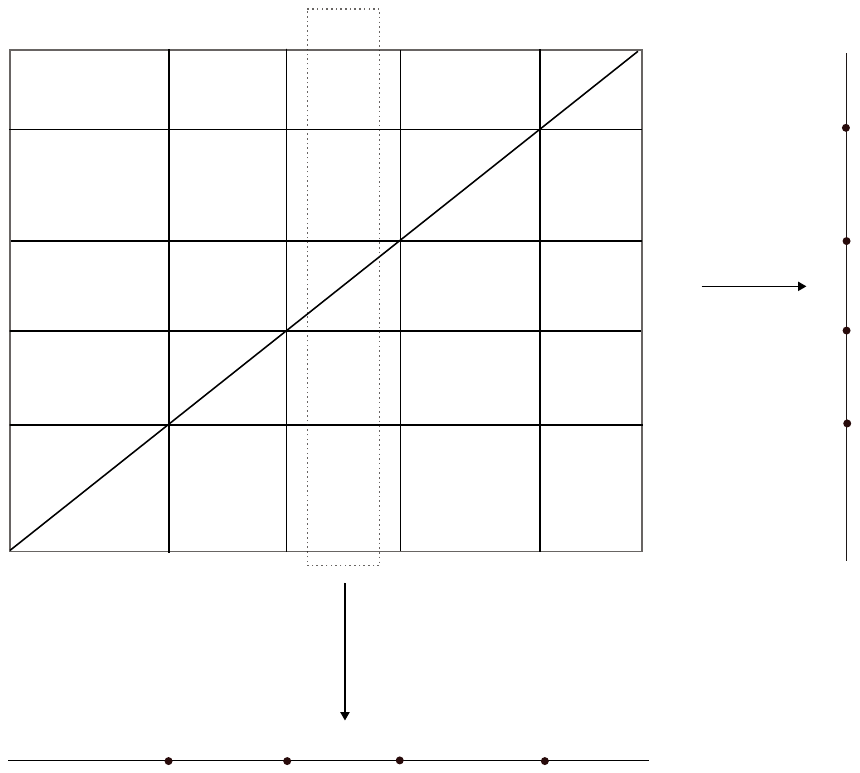}
\put(30,8){$x_1$}
\put(39,8){$x_2$}
\put(48,8){$x_3$}
\put(60,8){$\infty$}
\put(45, 19){$\pi_2\circ j$}
\put(86,37){$x_1$}
\put(86,44){$x_2$}
\put(86,52){$x_3$}
\put(86,61){$\infty$}
\put(76, 49){$\pi_1$}
\put(20, 63){$U$}
\put(15, 9){$X$}
\put(84, 69){$X$}
\put(25,31){$\Delta$}
\end{overpic}	
}
\caption{A depiction of $U$ with its two projections. The region enclosed in a dashed box is depicted in \autoref{figure:zoomed}. In general, the local system $\alpha^*\chi$ has nontrivial local monodromy about the diagonal $\Delta$, $\pi_1^{-1}(\infty), (\pi_2\circ j)^{-1}(\infty)$, and the local system $\pi_1^*\mathbb{W}$ has non-trivial local monodromy about $\pi_1^{-1}(x_i)$, $\pi_1^{-1}(\infty)$.}
\label{figure:mc}
}
\end{figure}
\subsection{Basic properties}
\begin{proposition}\label{prop:basic-props}
	If $\mathbb{W}$ is irreducible and $\chi$ is non-trivial, then
	\begin{enumerate}
		\item $MC_\chi(\mathbb W)$ is irreducible, and
		\item $MC_\chi(MC_{\chi^{-1}}(\mathbb{W}))=\mathbb{W}$.
	\end{enumerate}
	If $\mathbb{W}$ underlies a complex variation of Hodge structure and $\chi$ is unitary, then $MC_\chi(W)$ underlies a complex variation of Hodge structure as well.
\end{proposition}
\begin{proof}
	(1) and (2) are \cite[Theorem 2.9.7]{katz-rigid}. The statement about complex variations is part of \cite[Proposition 3.1.1]{dettweiler-sabbah}.
\end{proof}
We recall the effect of middle convolution on local monodromies. By $J(\lambda,
\ell)$ we mean the $\ell\times\ell$ Jordan block with generalized eigenvalue
$\lambda$, so, e.g.,~$J(\lambda, 2)$ is the matrix $$\begin{pmatrix} \lambda & 1 \\ 0& \lambda \end{pmatrix}.$$
\begin{proposition}[{\cite[Corollary 3.3.6]{katz-rigid}, \cite[Lemma 5.1]{dettweiler-reiter-painleve}}] \label{prop:local-monodromies} Suppose $\chi$ is a rank one local system on $\mathbb{G}_m$ with monodromy $\lambda\neq 1$ around $0$. The functor $MC_{\chi}$ changes the local monodromies of a local system $\mb{W}$ as follows:
\begin{enumerate}
\item 
At each finite puncture $x_i\in D\setminus \{\infty\}$, each Jordan block $J(\beta, \ell)$ of the local monodromy of $\mb{W}$ contributes $J(\beta\lambda, \ell')$ to the Jordan canonical form of the local monodromy of $MC_\chi(\mb{W})$ at $x_i$ where 
\[
\ell' =\begin{cases}
\ell \ \ \ \ \  \ \ \ \ \text{if} \ \beta\neq 1, \lambda^{-1}\\
\ell-1 \ \ \text{if} \ \beta =1 \\
\ell+1 \ \ \text{if} \  \beta=\lambda^{-1}.
\end{cases}
\]
All other Jordan blocks are of the form $J(1,1)$.
\item At $\infty$, each $J(\beta, \ell)$ contributes $J(\beta\lambda^{-1}, \ell')$ where 
\[
\ell' =\begin{cases}
\ell \ \ \ \ \  \ \ \ \ \text{if} \ \beta\neq 1, \lambda \\
\ell+1 \ \ \text{if} \ \beta =1 \\
\ell-1 \ \ \text{if} \  \beta=\lambda.
\end{cases}
\]
All other Jordan blocks are of the form $J(\lambda^{-1},1)$. 
\end{enumerate}
\end{proposition}

\section{A modification of Katz's algorithm}\label{section:katz}
We now present a slight variant of Katz's middle convolution algorithm, specialized to the rank $2$ case. Katz was interested in classifying rigid local systems on $\mathbb{P}^1\setminus D$; he showed that his algorithm, applied to an irreducible rigid local system of rank $r>1$, will always produce a rigid local system of rank $r'<r$. By contrast, in our setting of rank $2$ MCG-finite local systems, the algorithm will not reduce the rank of our local system in general---this would in fact imply the local system is rigid. Rather, it will produce a local system with finite monodromy, possibly with much larger rank.
\subsection{The algorithm}\label{section:thealgorithm}
Let $X=\mathbb{P}^1\setminus D$, for $D=\{x_1, \cdots, x_{n-1}, \infty\}$ some reduced effective divisor, and let $\mathbb{V}$ be an irreducible rank $2$ local system on $X$. We assume the isomonodromic deformation of $\mathbb{V}$ to a nearby $n$-pointed curve underlies a non-unitary polarizable complex variation of Hodge structure, as in \autoref{prop:not-all-unipotent}.  Following \cite[\S5.2]{katz-rigid}, our variant of Katz's algorithm proceeds as follows:
\begin{enumerate}
	\item Tensor with a rank one local system $\mathbb{L}_1$ so that $\mathbb{V}'=\mathbb{V}\otimes \mathbb{L}_1$ does not have non-trivial scalar monodromy at any of the $x_i$.
	\item Tensor with a rank one local system $\mathbb{L}_2$ so that the local monodromy of $\mathbb{V}''=\mathbb{V}'\otimes \mathbb{L}_2$ at $x_i$ has $1$ as an eigenvalue for $i=1, \cdots, n-1$.
	\item If the the local monodromy of $\mathbb{V}''$ at infinity is unipotent, relabel the points of $D$ so that this is no longer the case; this is always possible by \autoref{prop:not-all-unipotent}. Let $\nu\neq 1$ be one of the eigenvalues of the local monodromy of $\mathbb{V}''$ at $\infty$. Let $\chi$ be the rank one local system on $\mathbb{G}_m$ with local monodromy $\nu$ at $0$ and local monodromy $\nu^{-1}$ at $\infty$.
	\item Set $\mathbb{U}=MC_\chi(\mathbb{V}'')$.
\end{enumerate}
Note that as $\mathbb{V}$ underlies a polarizable complex variation of Hodge structure, the rank one local systems $\mathbb{L}_1, \mathbb{L}_2, \chi$ are all unitary. If $\mathbb{V}$ has monodromy conjugate to a representation valued $GL_r(\mathscr{O}_K)$ for some number field $K$, then $\mathbb{L}_1, \mathbb{L}_2, \chi$ are valued in $\mathscr{O}_{K'}^\times$ for some finite extension $K'/K$, and $MC_\chi(\mathbb{V}'')$ has monodromy conjugate to a subgroup of $GL_{r'}(\mathscr{O}_{K'})$.

In the rigid local systems setting, Katz iterates his procedure; we will halt after a single pass. Otherwise our variant of Katz's algorithm differs from his only in the relabeling in step (3), which is unnecessary in the rigid setting; we also have opted to avoid his use of the language of perverse sheaves.
\subsection{Unitarity of monodromy}
Throughout this section we assume $n\geq 3$. We will now begin preparations to show that Katz's algorithm, applied to a motivic rank $2$ local system on $\mathbb{P}^1\setminus D$ with $D$ generic and with infinite order at one or more punctures (or more generally, with no unitary Galois conjugates), produces a local system with finite monodromy.

With notation as in \autoref{section:notation}, let $(\pi: \mathscr{C}=\mathbb{P}^1\times \mathscr{M} \to \mathscr{M}, s_1, \cdots, s_{n+1})$ be a versal family of $(n+1)$-pointed projective lines,  let $\pi^\circ:\mathscr{C}^\circ\to \mathscr{M}$ be the associated family of punctured projective lines, and let $j: \mathscr{C}^\circ\hookrightarrow \mathscr{C}$ be the natural inclusion. Let $X$ be a general fiber of $\pi^\circ$, say over $m$, $\overline{X}=\mathbb{P}^1$ the smooth compactification, and $D=\overline{X}\setminus X$, $D=\{x_1, \cdots, x_{n+1}\}$.

We now perform the key computation in the proof:
\begin{proposition}[Key computation]\label{prop:key-computation}
Let $\mathbb{V}$ be a rank $2$ $\mathbb{C}$-local system on $\mathscr{C}^\circ$ such that $\mathbb{V}|_{X}$ is irreducible. Suppose that $\mathbb{V}$ underlies a polarizable complex variation of Hodge structure with $h^{1,0}=h^{0,1}=1$, and $h^{p,q}=0$ for $(p,q)\neq(1,0), (0,1)$. Suppose further that for $i=1, \cdots, n$, the local monodromy of  $\mathbb{V}|_{C^\circ}$ at $x_i$ has $1$ as an eigenvalue. Then $R^1\pi_*j_*\mathbb{V}$ has unitary monodromy.
\end{proposition}
\begin{proof}
Let $(\mathscr{E}, \nabla)$ be the logarithmic flat vector bundle on $\overline{X}$ obtained as the Deligne canonical extension of $\mathbb{V}|_X\otimes \mathscr{O}_X$ to $\overline{X}$. Let $F\subset \mathscr{E}$ be 
the non-trivial piece of the Hodge filtration; $F$ is a line bundle by hypothesis. Set $\overline{\mathbb V}$ to be the complex conjugate of $\mathbb V$, and $(\overline{\mathscr{E}}, \overline{F}, \overline{\nabla})$ 
the corresponding filtered flat vector bundle on $\overline{X}$. Note that $\overline{\mathbb{V}}$ also satisfies the hypotheses of the Proposition.
	
	It is enough to show that the Hodge filtration on $R^1\pi_*j_*\mathbb{V}$, which is a priori of length three, in fact has length one. Indeed, in this case the polarization is a flat unitary metric.
	 
	We claim it suffices to show that $\mathscr{E}/{F}=\mathscr{O}(-1)$. 
Indexing the Hodge filtration $F^\bullet$ on $R^1\pi_*j_*\mathbb{V}$ so that it is supported in degrees $-1, 0, 1$, we have (see \cite[\S2.3]{dettweiler-sabbah} for example) that $$H^1(\mathscr{E}/{F})=\on{gr}^{-1}_{F^\bullet}R^1\pi_*j_*\mathbb{V}|_m.$$
	Now we have that $$\overline{\on{gr}^{-1}_{F^\bullet}R^1\pi_*j_*\mathbb{V}}|_m=\on{gr}^{1}_{F^\bullet}R^1\pi_*j_*\overline{\mathbb{V}}|_m.$$ In particular, if $\mathscr{E}/{F}=\overline{\mathscr{E}}/\overline{F}=\mathscr{O}(-1)$ and hence $\on{gr}^{-1}_{F^\bullet}R^1\pi_*j_*\mathbb{V}|_m=\on{gr}^{-1}_{F^\bullet}R^1\pi_*j_*\overline{\mathbb{V}}|_m=0$, the same is true for $\on{gr}^{1}_{F^\bullet}R^1\pi_*j_*\mathbb{V}|_m$, and we may conclude.
	
	We now show, via a calculation with parabolic bundles, that $\mathscr{E}/{F}=\mathscr{O}(-1)$. Let $$\theta: F\to \mathscr{E}/F\otimes \Omega^1(\log D)$$ be the $\mathscr{O}$-linear map induced by $\nabla$. By \autoref{prop:NAHT}(3), the bundle $$\on{gr}^\bullet_{F^\bullet}\mathscr{E}= F\oplus \mathscr{E}/F$$ with the Higgs field $$\begin{pmatrix} 0 & \theta \\ 0 & 0 \end{pmatrix}$$ is a parabolically stable Higgs bundle of slope zero. 
	
	As $F$ is a quotient of this Higgs bundle, it must have positive parabolic degree. Hence in particular, as it is a saturated subbundle of $\mathscr{E}$, it must be the maximal destabilizing subbundle of $\mathscr{E}$ (as a parabolic bundle). As $X$ is a general fiber of $\pi^\circ$, we have $$\underline{\Hom}(F_\star, (\mathscr{E}/F)_\star)_0=\mathscr{O}(-1)$$ by \autoref{prop:unobstructed-destabilizing-bundle}.
	
	Now let $0\leq \alpha_i\leq \beta_i < 1$ be the parabolic weights of $\mathscr{E}$ at $x_i$. By assumption $\alpha_i=0$ for $i=1, \cdots, n$. Let $$m=\sum_i (\alpha_i+\beta_i)=\alpha_{n+1} +\sum_i \beta_i.$$ 
	We have $\deg \mathscr{E}=-m$; thus it suffices to show $\deg F=-m+1.$ Note that $\deg F\geq -m+1$, as $$-\deg F-m=\deg \mathscr{E}/F\leq \on{par.deg.}(\mathscr{E}/F)_\star<0,$$ again by parabolic stability of $\on{gr}^\bullet_{F^\bullet}\mathscr{E}$. 
	
	So set $f=-\deg F$; we know $f\leq m-1$ and wish to show $f\geq m-1$. Let $S\subset \{1, \cdots, n\}$ be the set of $i$ such that $F$ has nonzero parabolic weight at $x_i$. As the parabolic degree of $F$ is positive, we must have that $|S|\geq f$. Hence $$\mathscr{O}(-1)=\underline{\Hom}(F_\star, (\mathscr{E}/F)_\star)_0\subset \underline{\Hom}(F, \mathscr{E}/F)(-D_S),$$ where $D_S=\cup_{i\in S} \{x_i\}$ has degree at least $f$. Hence we have 
	\begin{align*}
	-1 &\leq \deg 	\underline{\Hom}(F, \mathscr{E}/F)(-D_S)\\
	&= \deg \mathscr{E}/F-\deg F-\deg D_S\\
	&= 2f-m-\deg D_S\\
	&\leq f-m
	\end{align*}
	Thus $f\geq m-1$ as desired.
\end{proof}

\subsection{The main technical result}\label{section:main-result-proofs} We may now prove the main technical results of the paper.

Let $\pi^\circ: \mathscr{C}^\circ\to \mathscr{M}$ be a versal family of $n$-punctured projective lines. Let $m\in \mathscr{M}$ be a general point and let $X$ be the fiber of $\pi^\circ$ over $m$.
\begin{proposition}\label{prop:unitary-MC}
	Let $\mathbb{V}$ be a rank $2$ $\mathbb{C}$-local system on $\mathscr{C}^\circ$ such that $\mathbb{V}|_X$ is irreducible. Suppose that $\mathbb{V}$ underlies a polarizable complex variation of Hodge structure with $h^{1,0}=h^{0,1}=1$, and $h^{p,q}=0$ for $(p,q)\neq(1,0), (0,1)$. Then applying Katz's algorithm (\S4.1) to $\mathbb{V}$ yields a local system with unitary monodromy.
\end{proposition}
\begin{proof}
	Without loss of generality, we may assume $n\geq 3$. Indeed, if $n<3$, then $\pi_1(\mathbb{P}^1\setminus D)$ is abelian, so there are no irreducible local systems of rank $2$ over $\mathbb{P}^1\setminus D$.

	We first apply steps (1) and (2) of Katz's algorithm, obtaining a local system $\mathbb{V}''$ with no non-trivial scalar monodromies, and with $1$ as an eigenvalue at each point of $D\setminus\{\infty\}$. Let $\chi$ be as in step (3) of Katz's algorithm. The main idea of the proof is to perform middle convolution \emph{in families}.
	
	We consider $\mathscr{U}=\mathscr{C}^\circ\times_\mathscr{M} \mathscr{C}^\circ\setminus \Delta$, where $\Delta$ is the diagonal, and let $j: \mathscr{U}\hookrightarrow \mathbb{P}^1\times \mathscr{C}^\circ$ be the natural inclusion. Let $\pi_1: \mathscr{U}\to \mathscr{C}^\circ$ be the projection onto the first coordinate, and $\pi_2:\mathbb{P}^1\times \mathscr{C}^\circ\to \mathscr{C}^\circ$ the projection onto the second coordinate. Note that $\pi_2\circ j$ is a versal family of $(n+1)$-punctured projective lines. As in \autoref{section:mc}, there is a map $\alpha: \mathscr{U}\to \mathbb{G}_m$ given fiberwise by $(x, y)\mapsto x-y$.
	
	We claim the versal family $\pi_2\circ j$, with the local system $\pi_1^*\widetilde{\mathbb{V}}\otimes \alpha^*\chi$, satisfies the hypotheses of \autoref{prop:key-computation}. This suffices because the restriction of $R^1\pi_{2*}j_*(\pi_1^*\widetilde{\mathbb{V}}\otimes \alpha^*\chi)$ to $X$ is precisely the output of Katz's algorithm. We need to check that the restriction of $\pi_1^*\widetilde{\mathbb{V}}\otimes \alpha^*\chi$ to a fiber of $\pi_2\circ j$ satisfies:
	\begin{enumerate}
		\item $h^{1,0}=h^{0,1}=1$, and
		\item has $1$ as an eigenvalue of at least $n$ of the $n+1$ local monodromy matrices.
	\end{enumerate} 	
	
\begin{figure}[h!]\label{figure:zoomed}
\centering{
\resizebox{80mm}{!}{
\begin{overpic}[trim={2cm 15cm 9cm 0}, clip]{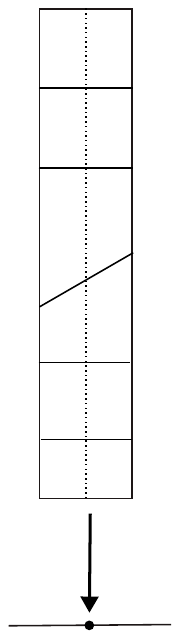}
\put(30,8){$y$}
\put(38, 31){$(1, \lambda_1)$}
\put(38, 41){$(1, \lambda_2)$}
\put(38, 53){$(\nu ,\nu)$}
\put(38, 63){$(1, \lambda_3)$}
\put(38, 72){$(1, \lambda_\infty)$}
\put(21, 31){$x_1$}
\put(21, 41){$x_2$}
\put(21, 46){$\Delta$}
\put(21, 63){$x_3$}
\put(21, 72){$\infty$}
\put(32, 17){$\pi_2\circ j$}
\end{overpic}	
}
\caption{A neighborhood of the fiber of $\pi_2\circ j: X\times X\setminus \Delta\to X$ over $y$, with eigenvalues of the local monodromy matrices of $\pi_1^*\widetilde{\mathbb V}\otimes \alpha^*\chi$ indicated on the right of the diagram.}\label{figure:zoomed}
\label{figure:Hodge-diagram}
}
\end{figure}
	
	Point (1) above is immediate from the fact that the same is true for $\mathbb{V}$, by hypothesis; tensoring with unitary rank one local systems cannot change this property. (As $\mathbb{V}$ underlies a polarizable complex variation and hence the eigenvalues of its local monodromy matrices have absolute value $1$, the local systems $\mathbb{L}_1, \mathbb{L}_2, \chi$ are necessarily unitary.) 
	
	Point (2) follows from the choices in Katz's algorithm, as we now explain; the reader may find it useful to refer to \autoref{figure:zoomed}. It suffices to check point (2) on any fiber of $\pi_2\circ j$. Consider $y\in X\subset \mathscr{C}^\circ$; by construction $X'=(\pi_2\circ j)^{-1}(y)$ is isomorphic to $X\setminus \{y\}$; for notational convenience, we choose this fiber. The local monodromy of $\pi_1^*\widetilde{\mathbb{V}}\otimes \alpha^*\chi|_{X'}$ at all points of $D$ except $\infty$ has $1$ as an eigenvalue by step (2) of Katz's algorithm. $X'$ has two remaining punctures: $y$ and $\infty$. The local monodromy at $y$ is scalar, so we must show the local monodromy at $\infty$ has $1$ as an eigenvalue. But this is immediate from our choice of $\chi$ in step (3) of Katz's algorithm.
\end{proof}

\begin{theorem}\label{thm:main-technical-thm}
	Let $D\subset \mathbb{P}^1$ be a generic effective divisor of degree $n$ containing $\infty$, and let $\mathbb{V}$ be a rank $2$ $\mathscr{O}_K$-local system on $X=\mathbb{P}^1\setminus D$ with non-scalar monodromy at $\infty$, such that for each embedding $\iota: \mathscr{O}_K\hookrightarrow \mathbb{C}$, $\mathbb{V}\otimes_\iota \mathbb{C}$ is irreducible and underlies a non-unitary polarizable complex variation of Hodge structure. Then applying Katz's algorithm (\autoref{section:thealgorithm}) to $\mathbb{V}$ yields a local system with finite monodromy.
\end{theorem}
\begin{proof}
Without loss of generality, we may assume $n\geq 3$. Indeed, if $n<3$, then $\pi_1(\mathbb{P}^1\setminus D)$ is abelian, so there are no irreducible local systems of rank $2$ over $\mathbb{P}^1\setminus D$.

	By Corlette-Simpson \cite[Theorem 2 and \S9]{corlette-simpson}, $\mathbb{V}$ is of geometric origin, and hence, spreading out, there exists a versal family of $n$-punctured projective lines $\mathscr{C}^\circ \to \mathscr{M}$, containing $X$ as a fiber, with $\mathscr{C}^\circ\subset \mathbb{P}^1\times\mathscr{M}$, and an $\mathscr{O}_K$-local system $\widetilde{\mathbb{V}}$ on $\mathscr{C}^\circ$ restricting to $\mathbb{V}$ on $X$.
	
	Choosing $\mathbb{L}_1, \mathbb{L}_2, \chi$ as in Katz's algorithm (\autoref{section:thealgorithm}), we have by construction that $\mathbb{L}_1, \mathbb{L}_2, \chi$ are valued in $\mathscr{O}_{K'}^\times$, for $K'$ some finite extension of $K$.  Now by hypothesis, for each embedding $\iota: \mathscr{O}_{K'}\hookrightarrow \mathbb{C}$, $\widetilde{\mathbb{V}}\otimes_\iota \mathbb{C}$ satisfies the hypotheses of \autoref{prop:unitary-MC} (as it is non-unitary). Hence for each $\iota$, the output of Katz's algorithm (\autoref{section:thealgorithm}) is unitary and defined over $\mathscr{O}_{K'}$, whence it has finite monodromy by e.g.~\cite[Lemma 7.2.1]{landesman2022geometric}. 
\end{proof} 
\begin{proposition}[Output of Katz's algorithm] \label{prop:katz-output}
	With hypotheses as in \autoref{thm:main-technical-thm}, let $D'\subset D$ be the set of points at which $\mathbb{V}$ has non-scalar monodromy. Then the local system $\mathbb{U}$ produced by Katz's algorithm (\autoref{section:thealgorithm}) has the following properties:
	\begin{enumerate}
		\item $\rk \mathbb{U}=|D'|-2$
		\item The monodromy group of $\mathbb{U}$ is a finite complex reflection group.
		\item The local monodromy of $\mathbb{U}$ at each $x_i\in D'\setminus \{\infty\}$ is a pseudoreflection. At each point of $D\setminus D'$, the local monodromy is trivial.
		\item The local monodromy of $\mathbb{U}$ at $\infty$ is $\nu^{-1} R$, where $\nu^{-1}$ is the local monodromy at $\infty$ of the local system $\chi$ on $\mathbb{G}_m$ chosen in Katz's algorithm, and $R$ is a pseudoreflection.
	\end{enumerate}
\end{proposition}
\begin{proof}
	Point (1) is immediate from the computation \cite[Corollary 3.3.6]{katz-rigid}
	of Katz of the rank of the middle convolution (see also \cite[Theorem
	2.2]{dettweiler-reiter-painleve}). Point (2) follows from (3) and
	\autoref{thm:main-technical-thm}. Point (3) follows directly from
	\autoref{prop:local-monodromies}(1). Finally, (4) follows from \autoref{prop:local-monodromies}(2), noting that by \autoref{thm:main-technical-thm}, the local system $\mathbb{U}$ has finite monodromy and hence all Jordan blocks have size $1$.
\end{proof}
Note that middle convolution is invertible by \autoref{prop:basic-props}, so \autoref{thm:main-technical-thm} tells us that $\mathbb{V}=MC_{\chi}(\mathbb{U})\otimes\mathbb{L}$, for appropriate rank one local systems $\chi$ on $\mathbb{G}_m$ and $\mathbb{L}$ on $X$, and $\mathbb{U}$ a local system on $X$ satisfying the conclusions of \autoref{prop:katz-output}.
\begin{remark}
The proof of \autoref{prop:unitary-MC} in fact shows that applying Katz's algorithm (\autoref{section:thealgorithm}) to irreducible \emph{supra-maximal} local systems \cite{deroin2019supra} yield 
a local systems with unitary monodromy. 
In our language, irreducible supra-maximal local systems have the following properties: they are rank $2$ local systems with trivial determinant, underlying a polarizable $\mathbb{C}$-VHS $(F, \mathscr{E}, \nabla)$ such that $\underline{\Hom}(F_\star, (\mathscr{E}/F)_\star)_0=\mathscr{O}(-1)$ (combining \cite[Theorem 5]{deroin2019supra} and \autoref{prop:unobstructed-destabilizing-bundle}). This answers a question of Deroin and Tholozan; we are grateful to them for pointing out this consequence of our work. This gives a different proof of the compactness of the space of supra-maximal representations, which is  \cite[Corollary 3]{deroin2019supra}.
\end{remark}
\subsection{Proofs of \autoref{thm:main-thm-intro} and \autoref{cor:main-classification}}\label{section:intro-proofs}
\begin{proof}[Proof of \autoref{thm:main-thm-intro}]
Let $\mathbb{V}$ be a local system as in the statement of \autoref{thm:main-thm-intro}. It suffices by \autoref{thm:main-technical-thm} and \autoref{prop:katz-output} to show that $\mathbb{V}$ satisfies the hypotheses of \autoref{thm:main-technical-thm}. As $\mathbb{V}$ is of geometric origin, its local monodromy matrices are quasi-unipotent, so the assumption that one of them has infinite order implies it is not diagonalizable. Hence no Galois conjugate of $\mathbb{V}$ can be unitary, as desired.

Now \autoref{thm:main-technical-thm} and \autoref{prop:katz-output} tell us that applying Katz's algorithm to $\mathbb{V}$ yields a rank one finite order local system $\mathbb{L}$ on $\mathbb{P}^1\setminus D$ and a non-trivial finite order rank one local system $\chi'$ on $\mathbb{G}_m$ such that $\mathbb{U}:=MC_{\chi'}(\mathbb{V}\otimes \mathbb{L})$ has monodromy a finite complex reflection group. Hence we have $\mathbb{V}=MC_{\chi}(\mathbb{U})\otimes \mathbb{M}$, where $\mathbb{M}=\mathbb{L}^{-1}$ and $\chi=(\chi')^{-1}$.
\end{proof}
\begin{proof}[Proof of \autoref{cor:main-classification}]
	We recall the setup from the introduction. Let $$\rho:
	\pi_1(\Sigma_{0,n})\to \SL_2(\mathbb{C})$$ be a Zariski-dense, MCG-finite representation, with at least one local monodromy matrix of infinite order. Without loss of generality we may assume $n\geq 3$, as if $n<3$, no Zariski-dense $\rho$ exist, since $\pi_1(\Sigma_{0,n})$ is abelian.
	
	By \cite[Corollary 2.3.5]{landesman2022geometric}, there exists a versal family of $n$-punctured projective lines $$\pi^\circ: \mathscr{C}^\circ\to \mathscr{M},$$ and a local system $\mathbb{V}$ on $\mathscr{C}^\circ$, such that the restriction of $\mathbb{V}$ to a fiber of $\pi^\circ$ has monodromy given by $\rho$. Suppose $\mathbb{V}$ does not have quasi-unipotent local monodromy at infinity. Then by \cite[Theorem A]{loray-etc}, $\mathbb{V}$ is pulled back through some map $\mathscr{C}^\circ \to Y$, with $Y$ a Deligne-Mumford curve, that is, we are in case (1) of the corollary, classified by Diarra \cite[\S3-\S5]{diarra}.
	
	Thus we may suppose $\mathbb{V}$ has quasi-unipotent monodromy at infinity. In this case, if $\mathbb{V}$ is not pulled-back through a map to a Deligne-Mumford curve, it is, by \cite[Theorem 2]{corlette-simpson}, of geometric origin. Hence its restriction to a general fiber of $\pi^\circ$ satisfies the hypotheses of \autoref{thm:main-thm-intro}, and we may conclude.
\end{proof}

\section{Finite complex reflection groups}\label{section:fcrg}
In this section we analyze the finite complex reflection groups that may arise from Katz's algorithm \autoref{section:thealgorithm}, using the Shephard-Todd classification. 
\subsection{Shephard-Todd classification}
When we refer to a complex reflection group, we will always view it as realized on a complex vector space $V$; we generally denote this representation by $\rho$, and by the trace field of $G$ we mean the field generated by the character values of $\rho$. We recall the following examples:
\begin{definition}
Let $m, N$ be positive integers. An $N\times N$ matrix is \emph{monomial} if it has precisely one non-zero entry in each row and column. The group $G(m, 1, N)\subset \GL_N(\mb{C})$ is the group of monomial matrices whose non-zero entries   are all $m$-th roots of unity. 

For a positive  integer $p$ dividing $m$, $G(m,p,N)\subset G(m,1,N)$ is the subgroup consisting of the matrices such that the product of all non-zero entries is an $m/p$-th root of unity. The representation $G(m,p,N)\subset \GL_N(\mb{C})$ is irreducible except for  $G(1,1,N)$ and $G(2,2,2)$ (see \cite[Proposition 2.10]{lehrertaylor}). 
\end{definition}
\begin{remark}\label{rmk:weyl-groups-Gmpn}
Note that $G(1,1,n)=W(A_n), G(2,1,n)=W(B_n)=W(C_n), G(2,2,n)=W(D_n)$, the Weyl groups of the infinite series of Dynkin diagrams, and $G(m,m,2)$ is the dihedral group of order $2m$. 
\end{remark}
We recall the Shephard-Todd classification of complex reflection groups. 

\begin{theorem}[{\cite{shephard-todd}, \cite[Theorem 8.29]{lehrertaylor}}]
Suppose $G\subset \GL_N(\mb{C})$ is an irreducible complex reflection group. Then it belongs to one of the following classes (some of which have overlaps):
\begin{enumerate}
\item $G=G(m,p,N)$ for some $m, p$;
\item $N=2$ and $G$ is one of 19 examples;
\item $G$ is the Weyl group of a semisimple Lie algebra (which we refer to as \emph{classical Weyl groups}), and $G\subset \GL_N(\mb{C})$ is the standard representation, i.e. the representation on the corresponding complexified root lattice;
\item $G$ is one of the following:
\[
W(H_3), W(J_3^{(4)}), W(J_3^{(5)}), W(L_3), W(M_3), W(H_4), W(L_4), W(N_4), W(O_4), W(K_5), W(K_6).
\]
\end{enumerate}
For the last class, $W(X)$ denotes the Weyl group of the complex root system $X$, and the subscript refers to the rank of the complex reflection group, i.e. the dimension of the vector space on which it is realized.
\end{theorem}
Our main goal in \autoref{section:fcrg} is to prove:
\begin{corollary}\label{gmp5}
For $N\geq 5$, the complex reflection group $G(m,p,N)\subset \GL_N(\mb{C})$ does not arise as the output of the Katz algorithm of \autoref{thm:main-technical-thm}; the same is true for $W(J_3^{(4)}), $ $W(H_4),$ $W(K_5),$ $W(K_6),$ $W(E_6),$ $W(E_7),$ $W(E_8)$. In particular, by the Shephard-Todd classification, there are no irreducible rank two motivic local systems $\mb{V}$ on $\mb{P}^1\setminus D$ with
\begin{enumerate}
\item no unitary Galois conjugates, and 
\item  with $D$ generic, $|D'|\geq 7$, where $D'\subset D$ denotes the subset with non-scalar local monodromies.
\end{enumerate}
\end{corollary}
We give the proof in \autoref{section:D-large}, along with the proof of \autoref{cor:finite-orbits-7}.
\subsection{Weyl groups}
\begin{corollary}\label{cor:weyl-group}
With notation as in \autoref{thm:main-technical-thm}, suppose $\mb{U}$, with
monodromy a classical Weyl group $G\subset \GL_N(\mb{C})$ for $N\geq 2$,  is the output of Katz's algorithm applied to the rank two local system $\mb{V}$ satisfying the hypotheses of \autoref{thm:main-technical-thm}. Then, after possibly  twisting by a rank one local system, $\mb{V}$ has coefficients in $\mb{Q}$, and is pulled back from the  modular curve $X(1)$.
\end{corollary}
\begin{proof}
By \autoref{prop:katz-output} and \autoref{prop:local-monodromies}, 
the local monodromies of $\mb{U}$ have eigenvalues
\begin{enumerate}
\item $\{\lambda \beta_i, 1, \cdots, 1\}$ at a finite point  $x_i\in D'$, where $1, \beta_i$ are the generalized eigenvalues  of $MC_{\chi^{-1}}(\mb{U})$ at $x_i$,
\item $\{\lambda^{-1}\gamma, \lambda^{-1}, \cdots, \lambda^{-1}\} $ at $\infty$, where $\lambda, \gamma$ are the generalized eigenvalues of $MC_{\chi^{-1}}(\mb{U})$ at $\infty$. 
\end{enumerate}
We first show that $\lambda=-1$. Since we are assuming $G\subset \GL_N(\mb{C})$ is the  representation of a Weyl group on the corresponding root lattice, each local monodromy matrix of $\mb{U}$ is defined over $\mb{Z}$; 
by considering the trace of monodromy at $\infty$, we obtain
\[
\lambda^{-1}(\gamma+N-1)  \in \mb{Z}.
\]
Since $|\lambda|=1$ and $|\gamma| = 1$, by taking the norm, we have
\begin{equation}\label{square}
|\gamma+N-1|^2=1+(N-1)(\gamma+\bar{\gamma})+(N-1)^2 =m^2
\end{equation}
for some $m\in \mb{Z}$. This implies that $\gamma$ is quadratic over $\mb{Q}$, i.e. it is a $2, 3, 4,$ or $6$-th root of unity. Hence $\gamma+\bar{\gamma}=0, \pm 1$ or $\pm 2$, and if moreover $N\geq 3$ we must have $\gamma+\bar{\gamma}=\pm 2$, i.e. $\gamma=\pm 1$; the latter implies, when $N \geq 3$, that $\lambda\in \mb{Q}$, and since it is a non-trivial root of unity,  we have $\lambda=-1$.

It remains to treat the case $N=2$, in which case  $G$ is one of $W(A_3), W(B_2), W(G_2)$, and it is a straightforward check that $\lambda=-1$ in these cases also.

Therefore we have shown that $\chi$ corresponds to the local system on $\mb{G}_m$ with monodromy $-1$ around zero. Since $\mb{U}$ is in fact a $\mb{Z}$-VHS, the same is true of $MC_{\chi^{-1}}(\mb{U})$; in other words, the latter is a rank two polarizable $\mb{Z}$-VHS of type $(1,1)$, and so comes from a family of elliptic curves, as required.
\end{proof}

\subsection{$G(m, p, N)$ for $N\geq 5$}
Let $S_N$ denote the $N$-th symmetric group. We first recall that there is a natural projection map $\pi: G(m,p,N)\rightarrow S_{N}$: we view $S_N\subset \GL_N(\mb{C})$ as the monomial matrices all of whose non-zero entries are $1$, and $\pi$ sends a matrix in $G(m,p,N)$ to $S_N$ by replacing each non-zero entry by $1$. 

\begin{proposition}[{\cite[Lemma 2.8]{lehrertaylor}}]\label{prop:classifypseudoref}
Suppose $\gamma \in G(m,p,N)$ is a pseudoreflection. Then $\pi(\gamma)$ is either trivial or a transposition.
\end{proposition}

\begin{lemma}\label{lemma:combinatorics}
	For $N\geq 5$, there are no collections of $N+2$ elements $s_1, \cdots , s_{N+2} \in G(m,p,N) $ such that 
\begin{itemize}
\item $s_1\cdots s_{N+2}=e$,
\item $s_i$ is a pseudoreflection for $i=1, \cdots , N+1$,
\item $s_{N+2}$ is a scalar multiple of a pseudoreflection (the pseudoreflection is not assumed to be in $G(m,p,N)$), and 
\item $s_1, \cdots , s_{N+2}$ generate $G(m,p,N)$.
\end{itemize}
\end{lemma}

\begin{proof}
We proceed by contradiction, so let us assume such a collection  $s_i \in G(m,p,N)$ exists. 
For $i=1, \cdots , N+2$, denote by  $t_i$ be the elements $\pi(s_i)\in S_{N}$. Note that each $t_i$ is either the trivial element or a transposition. This is clear for $i=1, \cdots, N+1$ so it remains to check it for $t_{N+2}$. By our assumption $s_{N+2}=\mu R$ for $\mu$ a root of unity and $R\in \GL_N(\mb{C})$ a pseudoreflection;  the group generated by $s_1, \cdots, s_{N+1}, R$ is certainly contained in $G(m',1,N)$ for some $m'$, and we may apply \autoref{prop:classifypseudoref} to conclude. 

The product of the $t_i$'s is then trivial. 
 Note that these transpositions must act transitively on $\{1,2,...,N\}$, since otherwise the group generated by the $s_i$'s is reducible.

We discard the $t_i$ which are trivial; let $k\leq N+2$ be the number of transpositions remaining, and relabel the $t_i$ so that $t_1, \ldots, t_k$ are transpositions with
$t_1 \cdots t_k = e$, and generating a transitive subgroup of $S_N$.
We refer to the sequence $(j,t_k(j), t_{k-1}t_k(j),
\ldots, t_1 \cdots t_k(j) = j)$ with consecutive repetitions removed as the {\em cycle
of $j$}.
For each $j \in \{1, \ldots, N\}$ let $c_j$ denote one fewer than the number of
elements in the cycle of $j$.
Informally, $c_j$ is the ``cycle length'' of $j$.
For $i \in \{1, \ldots, N\}$ let $a_i$ denote the number of $t_1, \ldots, t_k$
among which $i$ appears. Note that every $a_i\geq 2$, as otherwise $i$ would not be not be fixed by $t_1\cdots t_k$, and every $c_j \geq 2$, since the action is irreducible.

We first consider the case $N$ is odd.
Since a product of an odd number of transpositions is never the identity, we must have $k
\leq N+1$.
Since there are at most $N+1$ transpositions, $$\sum_{j=1}^N c_j =
\sum_{i=1}^N a_i \leq  2(N+1),$$ as each transposition contributes $2$ to both $\sum_j c_j$ and $\sum_i a_i$.

Since $N \geq 5$, the pigeonhole principle implies there are at least $2$ values of $j$ for which
$c_j = 2$.
Without loss of generality, we may assume $c_1 = 2$ and $c_2 = 2$.
The cycle of $1$ cannot be $1,2,1$, as then the cycle of $2$ would be $2,1,2$,
and the action would be reducible.
We can therefore assume the cycle of $1$ is $1,3,1$ and the cycle of $2$ is
$2,4,2$, meaning that the transpositions 
$(13),(31),(24),(42)$ appear among the $t_k$.
In order for $c_1=2$, we must either have (considering the cycle of $3$) $a_3
\geq 4$ or both $a_1 \geq 3$ and $a_3 \geq 3$. Similarly, we must either have $a_4
\geq 4$ or both $a_2 \geq 3$ and $a_4 \geq 3$.
This implies $a_1 + a_2 + a_3 + a_4 \geq 12$. Since
each $a_i \geq 2$ we find $12 + 2(N-4) \leq (a_1 + a_2 + a_3 + a_4) +\sum_{i=5}^N
a_i \leq 2N+2$, a contradiction.

The case $N$ is even and $N \geq 8$ is similar to the case $N$ is odd, where one argues by showing
there are at least $3$ values of $j$ with $c_j = 2$, and one obtains that $18
\leq \sum_{i=1}^6 a_i$, leading to a similar contradiction.
The case $N = 6$ follows from a brute-force verification.
\end{proof}

\begin{lemma}
\label{GmpNgeq5}
Suppose $\mathbb{V}$ satisfies	the hypotheses of \autoref{thm:main-technical-thm}. Then $G(m, p, N)$ does not arise as the output of applying Katz's algorithm to $\mathbb{V}$ if $N\geq 5$.
\end{lemma}
\begin{proof}
	Applying the Katz algorithm to $\mathbb{V}$ as in \autoref{thm:main-technical-thm}, the result is irreducible, and the local monodromies satisfy the conditions of \autoref{lemma:combinatorics} by \autoref{prop:katz-output}, which in turn implies that the resulting complex reflection group $G\subset \GL_N(\mb{C})$ cannot be  $G(m,p,N)$  for $N\geq 5$.
\end{proof}

\subsection{The groups $W(J_3^{(4)})$, $W(H_4)$, $W(K_5)$}
\begin{lemma}\label{lemma:h4argument}
Suppose $\mathbb{V}$ satisfies the hypotheses of \autoref{thm:main-technical-thm}. The groups $W(J_3^{(4)})$, $W(H_4)$, $W(K_5)$ do not arise as the output of applying Katz's algorithm to $\mathbb{V}$.
\end{lemma}
\begin{proof}
We follow the notation of \autoref{thm:main-technical-thm} and \autoref{prop:katz-output}. Let $G$ denote any of  $W(J_3^{(4)}), W(H_4), W(K_5)$, and $\rho$ its natural representation coming from the description as a complex reflection group. We use the following properties of $G$:
\begin{enumerate}
\item The trace field of $\rho$ is not totally real: it is $\mb{Q}(\sqrt{-7})$ for $W(J_3^{(4)})$, $\mb{Q}(\omega, \sqrt{5})$ for $W(H_4)$, and $\mb{Q}(\omega)$ for $W(K_5)$, where $\omega$ is a primitive third root of unity (see e.g. \cite[Table 1]{feit}),
\item all pseudoreflections are of order 2, and 
\item its center is $\mb{Z}/2$.
\end{enumerate}

Suppose $(G, \rho)$ arises from the Katz' algorithm, i.e. there is a local system $\mb{U}$ on a punctured $\mb{P}^1$ with monodromy $G$ and satisfying the properties listed in \autoref{prop:katz-output}. 

The monodromy $s_{\infty}$ at $\infty$ is given by $\lambda^{-1}R$ for some root of unity $\lambda\neq 1$ and pseudoreflection $R$. Note that $R$ is not a priori in $G$, but this is true in our case:  adding in the scalar element $\lambda$, we obtain a possibly larger irreducible finite complex reflection group $G'$ of  the same rank, but the Shephard-Todd classification shows that this must be $G$ itself. Indeed, we may rule out the case $G'=G(m,p,n)$ since this would imply $G$ is \emph{imprimitive}, and the only imprimitive finite complex reflection groups are the $G(m,p,n)$'s themselves---see \cite[Chapter 2]{lehrertaylor}.  For the exceptional ones, we simply rule them out by inspection---for example, their orders are not divisible by that of $G$. Therefore $R\in G$, and so $\lambda$ is in the center of $G$, which in turn implies $\lambda=-1$. 

Hence $\chi$ (as defined in \autoref{section:thealgorithm}) is the rank one local system on $\mb{G}_m$ with monodromy $-1$ around $0$. Applying $MC_{\chi^{-1}}=MC_{\chi}$ to $\rho$, the result is a rank two local system $\mb{V}$ with the same trace field as $\rho$.  On the other hand, by \autoref{prop:local-monodromies} and point (2) above,  the monodromy at each finite  point of $D$ is conjugate to $J(1,2)$, and therefore $\mb{V}$ has trivial determinant; $\mb{V}$ and all of its Galois conjugates underly polarizable variations of Hodge structure, and the unipotent local monodromies force all of these to have type $(1,1)$. In other words,  the monodromy of $\mb{V}$ and each of its Galois conjugates takes values in $\SL_2(\mb{R})$, and therefore $\mathbb{U}$ has totally real trace field, contradicting point (1) above.
\end{proof}
\subsection{The groups $W(J_3^{(5)}), W(N_4), W(O_4), W(K_6)$}
\begin{lemma}\label{N4etc}
Suppose $\mathbb{V}$ satisfies	the hypotheses of \autoref{thm:main-technical-thm}, and in addition $\mathbb{V}$ has local monodromy of infinite order at some point. Then $W(J_3^{(5)}), W(N_4), W(O_4), W(K_6)$ do not arise as the output of applying Katz's algorithm to $\mathbb{V}$.
\end{lemma}
\begin{proof}
Set $G$ to be one of $W(J_3^{(5)}), W(N_4), W(O_4), W(K_6)$. We recall the following facts about $G$:
\begin{enumerate}
	\item The trace field of the tautological representation of $G$ is not totally real. For $W(J_3^{(5)})$ it is $\mathbb{Q}(\omega, \sqrt{5})$, where $\omega$ is a primitive cube root of unity; for $W(N_4), W(O_4)$ it is $\mathbb{Q}(i)$; and for $W(K_6)$, it is $\mathbb{Q}(\omega)$ \cite[Table 1]{feit}.
	\item All pseudoreflections are of order $2$.
\end{enumerate}
The hypotheses of \autoref{thm:main-technical-thm} imply that $\mathbb{V}$ has quasi-unipotent local monodromy; hence the hypothesis that $\mathbb{V}$ has local monodromy of infinite order at some point implies there exists a point where the local monodromy of $\mathbb{V}$ is not semisimple. But as all pseudoreflections in $G$ have order $2$, this implies by \autoref{prop:local-monodromies} that $\chi$ (as defined in \autoref{section:thealgorithm}) is the rank one local system on $\mb{G}_m$ with monodromy $-1$ around $0$. Now we may proceed as in the proof of \autoref{lemma:h4argument}---$\mathbb{V}$, and hence $\mathbb{U}$, has totally real trace field, contradicting (1) above.
\end{proof}

\subsection{Further analysis of Mitchell's group}
We recall the following facts about Mitchell's group $W(K_6)$:
\begin{enumerate}
\item  the trace field of  $W(K_6)$ is $\mb{Q}(\omega)$, where $\omega$ is a primitive third root of unity \cite[Table 1]{feit}, and
\item all of its 126 pseudoreflections are of order 2, and its center is cyclic of order 6.
\end{enumerate}
Picking a basepoint $x\in \mb{P}^1\setminus \{x_1, \cdots, x_n\}$, suppose  that $(A_1, \cdots, A_n)$ represents a point of $X_n$. Then for any $i=1, \cdots, n-1$, the tuple $(A_1, \cdots, A_iA_{i+1}, \cdots, A_n)$ also has product equal to the identity, and we call it a \emph{coalescence} of $(A_1, \cdots, A_{n})$. We also call  $(A_nA_1, \cdots, A_{n-1})$ a coalescence of $(A_1, \cdots, A_{n})$; it is straightforward to see that any  coalescence of a MCG-finite tuple is again MCG-finite.
 
 \begin{lemma}\label{lemma:coalescehodge}
 Suppose $\mb{V}$ is an irreducible MCG-finite $\SL_2$-local system on an $n$-punctured $\mb{P}^1$ with $n\geq 8$, whose local monodromies around the punctures are non-scalar, and which moreover underlies a polarizable $\mb{C}$-VHS of type $(1,1)$. Suppose $\mb{V}'$ is a coalescence of $\mb{V}$ which stays irreducible. Then either 
 \begin{itemize}
 \item $\mb{V}'$, viewed as a local system on $\mathbb{P}^1\setminus \{y_1, \cdots, y_{n-1}\}$, also underlies a polarizable $\mb{C}$-VHS of type $(1,1)$, or 
 \item the Zariski closure of the monodromy of $\mb{V}'$ is contained in, up to conjugation, the infinite dihedral group 
 \[
 D_{\infty}=\bigg\{  \begin{pmatrix}
 c & 0 \\
 0 & c^{-1}
 \end{pmatrix} \bigg| c\in \mb{C}^{\times}   \bigg\} \cup 
 \bigg\{  \begin{pmatrix}
 0 & c \\
 -c^{-1} & 0
 \end{pmatrix} \bigg| c\in \mb{C}^{\times}   \bigg\}. 
 \]
  \end{itemize}
 \end{lemma}
 \begin{proof}

Note that $\mb{V}'$ cannot come from the pullback construction by \cite[Th\'eor\`eme 5.1]{diarra}. If it has Zariski-dense monodromy,  Corlette-Simpson implies that $\mb{V}'$ underlies a polarizable $\mb{C}$-VHS. On the other hand, the fiber $\mb{V}_x$ carries a Hermitian form of signature $(1,1)$ invariant under the group generated by $A_1, \cdots,  A_iA_{i+1}, \cdots, A_n$, inherited from the Hermitian form on $\mb{V}$, and since $\mb{V}'$ is assumed irreducible, there is a unique invariant form up to scaling. Therefore $\mb{V}'$ must also underly a polarizable $\mb{C}$-VHS of type $(1,1)$, as claimed.

On the other hand if $\mathbb{V}'$ does not have Zariski-dense monodromy, by the classification of algebraic subgroups (up to conjugation) of $\SL_2$ (see e.g. \cite[Theorem 4.29]{vanderputsinger}), since we are assuming $\mb{V}'$ is irreducible, the image of monodromy is either contained in $D_{\infty}$, or a finite irreducible subgroup  $G\subset \SL_2(\mb{C})$. The latter cannot happen since it would imply that $G$ has an invariant hermitian form of signature $(1,1)$, contradiction. Therefore the monodromy of $\mb{V}'$ is contained in $D_{\infty}$, as required.
 \end{proof}

\begin{lemma}\label{mitchell}
The group $G=W(K_6)$ does not arise as  the output of Katz's algorithm, applied to $\mathbb{V}$ satisfying	the hypotheses of \autoref{thm:main-technical-thm}.
\end{lemma}

\begin{proof}
Suppose the local system $\mb{U}$, with monodromy $G$, arises from the Katz algorithm from a MCG-finite, irreducible, $\SL_2$-local system $\mb{V}$, which we may take to be on $\mb{P}^1\setminus \{x_1, \cdots, x_7, \infty\}$, whose local monodromies around the eight punctures are non-scalar; recall that $\mb{V}$ and each of its Galois conjugates are assumed to underly a polarizable $\mb{C}$-VHS of type $(1,1)$, so that  \autoref{thm:main-technical-thm} applies. Note that all local monodromies of $\mb{V}$ must be semisimple, as otherwise the same argument as in \autoref{lemma:h4argument} applies: indeed, since the pseudoreflections in $G$ are all of order 2, the only way to get non-semisimple monodromies is if $\chi$ has monodromy $-1$, where $\chi$ is the character for the middle convolution in the Katz algorithm.

We will say that a $2\times 2$ matrix is \emph{neat} if it is regular semisimple, i.e. semisimple with distinct eigenvalues. Picking a basepoint and loops around $x_1, \cdots, x_7, \infty$, we  write $\mb{V}$ equivalently as a tuple of matrices $(A_1, \cdots, A_7, A_{8})$, each of which is neat.

\begin{claim*}
Coalescing the  punctures $x_j, x_{j+1}$, the obtained  local system $\mb{V}'$ cannot stay irreducible. 
\end{claim*}
\begin{proof}[Proof of claim]
Suppose $\mb{V}'$ is irreducible. By \autoref{lemma:coalescehodge}, either  $\mb{V}'$ has image in $D_{\infty}$, or  it   underlies a VHS of type $(1,1)$. We now rule out the $D_{\infty}$ case. Note that there is the sign homomorphism $\sgn: D_{\infty}\rightarrow \{\pm 1\}$, given by quotienting out the diagonal matrices. We must have $\sgn(A_i)=1$ for $i \neq j, j+1$: indeed, if $\sgn(A_i)=-1$, then $A_i$ has trace zero, and hence has eigenvalues $\pm \sqrt{-1}$. This may be ruled out  by analyzing the local monodromies of $MC_{\chi^{-1}}(\mb{U})$, combined with the fact that  $\chi$ must have monodromy a sixth root of unity. Since $A_1A_2\cdots A_8=\id$, the above implies $\sgn(A_jA_{j+1})=1$, which in turn implies that $\mb{V}'$ is reducible, a contradiction. 

Therefore $\mb{V}'$ and all of its Galois conjugates underlie VHS of type $(1,1)$. We may then apply the Katz algorithm to get as output an irreducible rank 5 complex reflection group; it must then be $G(m,p,5)$ for some $m,p$, or $W(K_5)$. The former is ruled out by \autoref{GmpNgeq5}, and the latter by \autoref{lemma:h4argument}.
\end{proof}

Hence we may assume that any coalescence of two points gives a reducible local system.

 Coalescing the points $x_i, x_{i+1}$, where $i=1, \cdots, 7,$ we obtain the tuple $(A_1, \cdots, A_iA_{i+1}, \cdots, A_{8})$. By the above, this tuple corresponds to  a reducible representation, i.e. a representation into the group of upper triangular matrices $B\subset \SL_2(\mb{C})$. Composing with the natural map $B\rightarrow \mathrm{Aff}(\mb{C})$, the latter being the group of affine transformations of $\mb{C}$,  we obtain  a representation into $\mathrm{Aff}(\mb{C})$.  We now apply the results of \cite{cousinmoussard} which classify  MCG-finite  representations into $\rm{Aff}(\mb{C})$. 
 
\textbf{Case 1}: Suppose for some $j=1, \cdots,7$, $A_jA_{j+1}$ is neat.  Now we apply \cite[Theorem 2.3.4]{cousinmoussard}.  As  all the matrices in the tuple $(A_1, \cdots A_jA_{j+1}, \cdots, A_{\infty})$ are neat,  loc.cit. implies that the only such local systems are on an $n$-punctured $\mb{P}^1$ for $n\leq 6$, or a sum of two characters; since we now have seven punctures, we must be in the latter case. Therefore, in this case,  we may assume that  $A_i$ is diagonal for any $i=1, \cdots, j-1, j+2, \cdots, \infty$, along with $A_jA_{j+1}$; in particular  these matrices commute with one another.
 
 Now we instead coalesce $x_{j+1}, x_{j+2}$ (where $A_{9}=A_1$) for the original tuple $(A_1, \cdots, A_{8})$. If $A_{j+1}A_{j+2}$ is also neat, then the same argument as above shows that $A_j$ commutes with $A_{j-1}$ (with $A_{-1}=A_8$), and hence $A_j$ is also diagonal, and hence so is $A_{j+1}$, implying $\mb{V}$ is reducible, contradiction. If $A_{j+1}A_{j+2}$ is not neat, then \cite[Theorem 2.3.3]{cousinmoussard} implies it is scalar, and hence $A_{j+1}$ is diagonal, again implying $\mb{V}$ reducible, which is a contradiction. 
 
 \textbf{Case 2}: therefore we may assume that $A_iA_{i+1}$
 is not neat, for any choice of $i=1, \cdots,  7$. On the other hand,  \cite[Theorem 2.3.3]{cousinmoussard} then implies that $A_iA_{i+1}=\pm1$ for each such $i$, which in turn implies that $\mb{V}$ is reducible, which is again a contradiction.    
\end{proof}

\subsection{Ruling out local systems on $\mathbb{P}^1\setminus D$, with $|D|$ large}\label{section:D-large}

We may now prove \autoref{gmp5} and \autoref{cor:finite-orbits-7}.

\begin{proof}[Proof of \autoref{gmp5}]
Recall from the statement of \autoref{gmp5} that $D'$ is the number of points at which $\mathbb{V}$ has non-scalar monodromy. By \autoref{prop:katz-output}, the result of applying our variant Katz's algorithm of \autoref{thm:main-technical-thm} to $\mathbb{V}$ is a local system $\mathbb{U}$ of rank $N=|D'|-2$, with monodromy group a finite complex reflection group. We now study the possible monodromy groups of $\mathbb{U}$.

\autoref{GmpNgeq5} rules out the group $G(m,p,N)$  for $N\geq 5$. All classical Weyl groups of rank at least $5$ are handled by the previous sentence and \autoref{rmk:weyl-groups-Gmpn}, except $W(E_6), W(E_7), W(E_8)$. That $W(E_6), W(E_7), W(E_8)$ do not appear follows from \autoref{cor:weyl-group}, 
as the corollary shows that classical Weyl groups give rise to ``pullback representations" with non-scalar monodromy at $|D'|=N+2$ points by \autoref{cor:weyl-group}.  But these do not exist for $|D'|\geq 7$ by \cite[Th\'eor\`eme 5.1]{diarra}.  Similarly, we have ruled out $W(J_3^{(4)}), W(H_4), W(K_5)$ in \autoref{lemma:h4argument} and $W(K_6)$ in \autoref{mitchell}.

The second claim now follows from the Shephard-Todd classification; we have ruled out all finite complex reflection groups of rank $N\geq 5$, and hence all motivic local systems as in the statement with non-scalar local monodromy at $|D'|=N+2\geq 7$ points.
\end{proof}

\begin{proof}[Proof of \autoref{cor:finite-orbits-7}]
This is immediate by combining \autoref{gmp5} with \cite[Th\'eor\`eme 5.1]{diarra}. Diarra's result rules out finite orbits of ``pullback type," and ours rules out local systems of geometric origin. As discussed in the introduction and in the proof of \autoref{cor:main-classification}, the results of \cite[Theorem 2]{corlette-simpson} and \cite[Theorem A]{loray-etc} show that these are the only two possibilities, completing the proof.
\end{proof}

\subsection{Examples}\label{section:examples}
\subsubsection{Dihedral groups}
As explained in \cite{lamlitt}, there is an infinite collection of finite orbits of the $\on{Mod}_{0,4}$-action on $X_4$ arising via middle convolution from rank $2$ local systems with dihedral monodromy. We    now give a brief explanation of this.  

For any $\lambda\in \mb{C}$ different from $0, 1$, let $X$ denote $\mb{P}^1\setminus \{0, 1, \lambda, \infty\}$, and let $\pi: E\rightarrow \mb{P}^1$ denote the double cover branched at $0, 1, \lambda, \infty$. Denote by $E^{\circ}$ the curve $\pi^{-1}(X)$, so that we have the \'etale map $\pi^{\circ}: E^{\circ} \rightarrow X$.

In loc.cit it was shown (see \cite[Proposition 4.2.2]{lamlitt}) that each irreducible  rank two local system $\mb{V}$ on $X$ of geometric origin and  satisfying a local monodromy condition $(\star)$ (see \cite[p.2]{lamlitt}) is of the form $MC_{\chi}(\pi^{\circ}_*\mb{L}|_{E^{\circ}})$; here $\mb{L}$ is a  torsion rank one local system on $E$, and  $\chi$ has monodromy $-1$. Note that $\pi^{\circ}_*\mb{L}|_{E^{\circ}}$ has monodromy a dihedral group; in other words $\mb{V}$ arises as the middle convolution of a local system with monodromy  a complex reflection group. An explicit description of the monodromy representation on these $\mathbb{V}$ is given in \cite[Example 1.1.7]{lamlitt}.

\subsubsection{$W(H_3)$}\label{h3}
The  group $W(H_3)$ is the group of symmetries of the icosahedron, so that $W(H_3)\simeq \mb{Z}/2\times A_5$, with the representation realizing it as a reflection group being either of the two (Galois-conjugate) rank 3 irreducible representations with non-trivial central character. The trace field of $W(H_3)$ is $\mb{Q}(\sqrt{5})$. 
\begin{enumerate}
\item The center of $W(H_3)$ is given by the $\mb{Z}/2$ factor, which acts by $-1$ on the corresponding three dimensional representation; we denote a generator of the $\mb{Z}/2$ by $-1$.
\item The pseudoreflections are precisely the products of $-1$ with the elements of $A_5$ which are products of two disjoint transpositions (i.e.~the conjugacy class of $(12)(34)$), so that there are 15 of such. \end{enumerate} 
Combining these, we see that each appearance of $W(H_3)$ as an output of Katz's algorithm in \autoref{thm:main-technical-thm}  corresponds to five elements $\sigma_1, \cdots ,\sigma_5\in A_5$ as in point (2) above such that $\sigma_1\cdots \sigma_5=e$. For example, we  can take the tuple $(\sigma_1, \sigma_2, \sigma_3, \sigma_4, \sigma_5)$ to be 
\[
((12)(34), (34)(51), (51)(23), (23)(45), (45)(12)).
\]

This produces, via middle convolution, a rank two  local system $\mb{V}$ with non-semisimple monodromy at all punctures. Since the trace field $\mb{Q}(\sqrt{5})$ is quadratic over $\mb{Q}$, $\mb{V}$ arises in the cohomology of  a family of abelian surfaces with real multiplication by $\mb{Q}(\sqrt{5})$ over a generic 5-punctured $\mb{P}^1$; moreover, this family has  (potentially) totally degenerate reduction at each puncture by Grothendieck's semistable reduction criterion. 

\subsubsection{$G(m,p,3)$}\label{gmp3}
Suppose the elements
\[
s_1=\begin{pmatrix}
 	   & a &  \\
 a^{-1} &    &   \\
 	   &    & 1
\end{pmatrix}, 
s_2=\begin{pmatrix}
 	   & b &  \\
 b^{-1} &    &   \\
 	   &    & 1
\end{pmatrix},
s_3=\begin{pmatrix}
 	1   &  &  \\
 	     &    & c  \\
 	     &   c^{-1} & 
\end{pmatrix},
s_4=\begin{pmatrix}
 	1   &  &  \\
 	     &    & d  \\
 	     &   d^{-1} & 
\end{pmatrix},
\]
of $G(m,p,3)$ generate the group and satisfy
\[
s_1s_2s_3s_4=\lambda^{-1} R
\]
with $R$ a pseudoreflection and $\lambda\neq 1$. Then the tuple $(s_1, s_2, s_3, s_4, \lambda R)$ corresponds to a local system on a 5-punctured $\mb{P}^1$, with monodromy $G(m,p,3)$; applying $MC_{\chi}$, with $\chi$ having monodromy $\lambda$, we obtain a rank two, irreducible, MCG-finite local system. It is straightforward to find examples of such elements $s_1, \cdots , s_4$; for example, any choice of roots of unity $a,b,c,d$ such that $a/b=d/c\neq bc/ad$ suffices.
\subsubsection{$W(D_4)$}\label{d4}
The elements \[
s_1=\begin{pmatrix}
     & -1 & & \\
 -1 &     & & \\
     &     & 1 & \\
     &     & & 1
\end{pmatrix}, 
s_2=\begin{pmatrix}
     & 1 & & \\
 1 &     & & \\
     &     & 1 & \\
     &     & & 1
\end{pmatrix},
s_3=\begin{pmatrix}
    1 &  & & \\
  &  1   & & \\
     &     &  & -1 \\
     &     & -1 & 
\end{pmatrix},
\]
\[
s_4=\begin{pmatrix}
    1 &  & & \\
  &    1  & & \\
     &     &  & 1 \\
     &     & 1 & 
\end{pmatrix},
s_5=\begin{pmatrix}
     &  & -1 & \\
  &   1  & & \\
 -1    &     &  & \\
     &     & & 1
\end{pmatrix},
s_6=\begin{pmatrix}
     &  & 1 & \\
  &   -1  & & \\
 1    &     &  & \\
     &     & & -1
\end{pmatrix},
\]
of $W(D_4)=G(2,2, 4)$
satisfy $s_1\cdots s_6=1$, and $-s_6$ is a pseudoreflection. The tuple $(s_1, \cdots, s_6)$ corresponds to a rank 4 local system on a 6-punctured $\mb{P}^1$ with monodromy $W(D_4)$; applying $MC_{\chi}$ where $\chi$ has monodromy $-1$, we obtain a rank two, irreducible, MCG-finite local system on the projective line minus six points, with non-semisimple monodromy at all $6$ points. This shows that \autoref{gmp5} and \autoref{cor:finite-orbits-7} are sharp.

\subsection{Table} We summarize some data regarding the finite complex reflection groups which may arise from middle convolution as above:
\begin{center}
\begin{tabular}{ |c|c| c|} 
 \hline
 Complex reflection group & Arising from Katz algorithm? & Geometric origin of $\SL_2$-local system \\
 \hline
 Dihedral groups & \cite{lamlitt} & AVs with RM by $\mb{Q}(\zeta+\zeta^{-1})$\\

 other rank 2 & see \cite{lisovyy-tykhyy} & many cases \\
 $G(m,p,3)$ & \autoref{gmp3}  &  many cases\\ 

 $G(m,p,N)$, $N\geq 5$ & None exists: \autoref{gmp5}  &  --\\
 	$W(H_3)$ 	& \autoref{h3} & AVs with RM by $\mb{Q}(\sqrt{5})$\\
	$W(J_3^{(4)}), W(H_4), W(K_5)$ & None exists: \autoref{lemma:h4argument} & -- \\
	$W(K_6)$ & None exists: \autoref{mitchell} & -- \\
	$W(D_4)$ & \autoref{d4} & Elliptic curves \\
	Classical Weyl groups & Pullback families &  Elliptic curves (by \autoref{cor:weyl-group})\\
 \hline
\end{tabular}
\end{center}
\begin{remark}
We leave open the  analysis for the cases $G(m,p,4), W(L_3), W(M_3), W(J_3^{(5)}), W(L_4)$,  $W(N_4), W(O_4), W(F_4)$, except for the special cases considered in \autoref{N4etc}, where $\mathbb{V}$ has a point with local monodromy of infinite order.
\end{remark}
\bibliographystyle{alpha}
\bibliography{bibliography-garnier}

\end{document}